\documentclass[11pt,oneside,reqno]{amsart}
\usepackage{mathrsfs}
\usepackage{amstext}
\usepackage{amsthm}
\usepackage{amssymb}
\usepackage{geometry}
\usepackage{setspace}
\usepackage%[bookmarks=false,
 {hyperref}

\makeatletter
\numberwithin{equation}{section}
\numberwithin{figure}{section}
\theoremstyle{plain}
\newtheorem{thm}{\protect\theoremname}[section]
\theoremstyle{plain}
\newtheorem*{claim*}{Claim}
\theoremstyle{remark}
\newtheorem{rem}[thm]{\protect\remarkname}
\theoremstyle{definition}
\newtheorem{defn}[thm]{\protect\definitionname}
\theoremstyle{plain}
\newtheorem{lem}[thm]{\protect\lemmaname}
\theoremstyle{plain}
\newtheorem{cor}[thm]{\protect\corollaryname}
\theoremstyle{remark}
\newtheorem*{acknowledgement*}{\protect\acknowledgementname}
\theoremstyle{definition}
\newtheorem{notation}[thm]{\protect\notationname}

\usepackage{amsthm}
\usepackage{mathrsfs}
\usepackage[noadjust]{cite}
\usepackage{enumitem}
	\setlist[itemize]{leftmargin=*}
	\setlist[enumerate]{leftmargin=*}
\usepackage{tikz}
\usetikzlibrary[patterns]
\usepackage{bbm}
\usepackage{stmaryrd}

\makeatother

\providecommand{\corollaryname}{Corollary}
\providecommand{\definitionname}{Definition}
\providecommand{\lemmaname}{Lemma}
\providecommand{\remarkname}{Remark}
\providecommand{\theoremname}{Theorem}
\providecommand{\acknowledgementname}{Acknowledgement}
\providecommand{\notationname}{Notation}

\newcommand{\T}{\mathbb{T}}
\newcommand{\Prob}{\mathbb{P}}
\newcommand{\Q}{\mathbb{Q}}
\newcommand{\Z}{\mathbb{Z}}

\begin{document}
\title[Sharp freezing time estimates for the subcritical FEP]{Sharp freezing time estimates for the subcritical Facilitated Exclusion Process}
\author{Oriane Blondel, Cl\'ement Erignoux, Seonwoo Kim and Sanha Lee}
\address{Oriane Blondel. Université Paris Cité and Sorbonne Université, CNRS, Laboratoire de Probabilités, Statistique et Modélisation, F-75013 Paris, France.}
\email{blondel@lpsm.paris}
\address{Cl\'ement Erignoux. Inria, Universite Claude Bernard Lyon 1, CNRS UMR5208, Institut Camille Jordan, F-69603 Villeurbanne, France.}
\email{clement.erignoux@inria.fr}
\address{Seonwoo Kim. Yonsei University, 50 Yonsei-ro, Seodaemun-gu, Seoul 03722, South Korea.}
\email{seonwookim@yonsei.ac.kr}
\address{Sanha Lee. Seoul National University, 1 Gwanak-ro, Gwanak-gu, Seoul 08826, South Korea.}
\email{sanha7139@snu.ac.kr}
\begin{abstract}
We investigate the exact transience time of the Facilitated Exclusion Process (FEP) on the one-dimensional torus with $N$ sites. The FEP exhibits an active/inactive phase transition at critical density $1/2$, such that in the subcritical density regime $(0,1/2)$, it becomes frozen after a finite time period---the transience time or freezing time. We first show that for the FEP starting from a Bernoulli product measure of marginal density $\rho \in (0,1/2)$, the transience time has exactly the scale of $\Theta(\log^3 N)$. Secondly, we prove that in the near-critical case $\rho \simeq 1/2 - N^{-\alpha}$ for $\alpha \in (0,1)$, the transience time is polynomial and has a scale of $N^{1 \wedge (2\alpha)}$. The key idea is to estimate the typical size of locally supercritical intervals of the initial distribution, which has order $\log N$ in the subcritical case and $N^{1 \wedge (2\alpha)}$ in the near-critical case. In the subcritical case this is enough, whereas in the near-critical case we need additional dynamical decorrelation inequalities to apply this static result to estimate the freezing time.
\end{abstract}

\maketitle
\tableofcontents{}

\section{\label{sec1}Introduction}

%\cosw{SW: What do you think of the title/abstract? Please modify as you wish if unsatisfied!}

%In the last twenty years, significant effort has been invested by the mathematics community into modelling and understanding the behavior of multi-phase materials. In this context, lattice gases have been heavily scrutinized, as they can model a wide variety of physical behavior, and in many cases are open to exact mathematical characterization. Constrained lattice gases, in particular, have been useful tools to model phase separation, as they exhibit radically different macroscopic behavior depending on whether the kinetic constraint is typically satisfied at the local level. Refer to the monograph \cite{HT25} for an intensive mathematical/physical review on this general topic and a full list of references.
%{\color{cyan}O: I don't understand what papers/results the two previous sentences allude to (starting at "open to exact math..."), we should have more precise statements as well as citations.} \cosw{SW: maybe Hartarsky--Toninelli helps?}

In this article, we are interested in the Facilitated Exclusion Process (FEP), which has been heavily studied in recent years \cite{BBCS18,BESS20,BES21,mappingESZ, fluctuationsFEP,GLS, WAFEP-BBS, DCES}, as a toy model with absorbing phase transition and phase separation \cite{RPSV}.
%\cosw{SW: what is phase separation?}
The macroscopic behavior of the model depends on the local density $\rho$: in the supercritical phase $\rho > 1/2$ it will be typically diffusive, while it will be frozen in the subcritical phase $\rho\leq 1/2$. At the microscopic level, phase separation only occurs after a transience time, before which active and frozen particles clusters coexist (see Section \ref{sec2.2} for a more rigorous formulation). {Understanding this transient behavior is key in the study of the FEP dynamics.}

The transient time of the FEP has been previously studied in several cases. Because of the kinetic constraint, particle trajectories in the FEP are heavily correlated, and because the FEP itself is not attractive, understanding its transient behavior is by no means trivial. In \cite{BESS20}, a  polylogarithmic bound $\mathcal O(\log^{32} N)$ is obtained, starting from a supercritical Bernoulli product measure on the ring. This is instrumental in deriving the FEP's supercritical hydrodynamic limit. Using similar techniques, \cite{BES21} estimates the freezing time in the subcritical phase, as well as the time to create well defined interfaces in a more general situation, e.g. starting from product measures fitting a suitable density profile. Both of these works use heavily a mapping of the FEP to an attractive zero-range process. One significant upside of the latter is that it can be constructed by sampling particle trajectories \emph{a priori} (as simple random walks), before adjusting the time between jumps in order to fit the zero-range interactions. The main drawback in this method is that the time change induced by the inter-particles interactions is not tightly controlled, and sharp time estimates are hard to obtain.

In \cite{EM24}, this issue is solved by mapping to another attractive process called SSEP with traps (SWT), in which individual particles jump at rate $1$. This mapping is explained in full detail in Appendix \ref{secA}. This idea allows to obtain sharp estimates on the transience time, starting from any configuration. The cutoff phenomenon can then be proved: the transience probability on the ring decays sharply in a vanishing window around time $t_N^\star:=N^2\log N/4\pi^2$, extending results from \cite{AC25}.

Our main results are two-fold.
First, we revisit the estimation of the transience time starting from a product measure, in order to obtain sharp estimates in the subcritical phase. More precisely, we show that the transience time starting from a Bernoulli product measure with parameter $\rho<1/2$ is indeed polylogarithmic, of exact order $C(\rho)\log^3 N$. See Theorem \ref{thm:main} for details.
%We also explore the divergence of the constant $C(\rho)$ as we approach criticality ($\rho\nearrow 1/2$).
The proof relies on the mapping to the SWT supplemented by the construction of non-interacting $\log N$-sized boxes whose boundaries  will never be crossed by particles due to the overwhelming presence of traps (corresponding to empty sites in the FEP).

Second, we consider the near-criticality case (also studied in \cite{GLS}, with a focus on the variance of the number of particles in large boxes), when the number of particles $K$ is roughly of order $N/2 - N^{1-\alpha}$ (cf. \eqref{eq:Kalpha}), such that the global density becomes $\rho \simeq 1/2 - N^{-\alpha}$ as $N \to \infty$. Then, we identify a certain \emph{dynamical phase transition}
%{\color{cyan}O: Why do we call this phase separation? Dynamical transition would make more sense to me, but maybe there is something even better?}
at $\alpha^\star := 1/2$ in the following sense. If $\alpha$ is big (i.e. $\alpha \ge 1/2$), then the transience time scales diffusively as $N^2$, whereas if $\alpha$ is small (i.e. $\alpha < 1/2$), then the transience time has a scale $N^{4\alpha}$, which is strictly smaller than $N^2$. We refer the reader to Theorem \ref{thm:main2} for the exact mathematical formulation.

The emergence of dynamical phase transition
%{\color{cyan}O: Same remark.}
in the near-criticality case comes from the competition between the global subcriticality and the local CLT fluctuation of the density. More precisely, if the subcriticality $N^{-\alpha}$ is small enough, then the local CLT fluctuation overcomes this subcriticality and allows macroscopic $\Theta(N)$-size intervals to be supercritical, so that they take a diffusive time scale $N^2$ to freeze. On the other hand, if $N^{-\alpha}$ is large, then the CLT fluctuation can only make $\Theta(N^{2\alpha})$-size intervals to be initially supercritical, thus the corresponding freezing time becomes roughly close to $N^{4\alpha}$. See Remark \ref{rem:phase-sep} for more discussion on this matter.

The rest of the article is organized as follows.
In Section \ref{sec2} we define the FEP, recall a few basic facts about its transient/frozen/ergodic states, and state our main results in the sub-critical (Theorem \ref{thm:main}) and near-critical (Theorem \ref{thm:main2}) cases. Sections \ref{sec3} and \ref{sec4} are devoted to the proof of Theorem \ref{thm:main}, while Sections \ref{sec5} and \ref{sec6} contain the proof of Theorem \ref{thm:main2}. The Appendix collects useful tools, notably a mapping between the FEP and the SWT, as well as standard estimates for random walks and large deviations for sums of i.i.d.\ geometric random variables.

For the proof of Theorem \ref{thm:main}, the key point is Lemma \ref{l4}, which allows us to identify from the initial configuration a condition under which separate regions never communicate. The rest of Section \ref{sec3} is devoted to estimating the size of these regions using standard random walk techniques. Section \ref{sec4} uses the inputs from Section \ref{sec3} to prove Theorem \ref{thm:main}.

Section \ref{sec5} contains two main results: the existence of supercritical regions and a decorrelation estimate for the restriction of the SWT dynamics in distant subboxes. Section \ref{sec6} collects these estimates and proves Theorem \ref{thm:main2}.

\section{\label{sec2}Model and results}

\subsection{\label{sec2.1}The model: Facilitated Exclusion Process}

Consider the discrete ring
\[
\mathbb{T}_{N}=\{1,2,\dots,N\}
\]
with $N+1=1$ and $0=N$. As an exclusion process, the state space for the FEP is defined as
\[
\Sigma_{N}:=\{0,1\}^{\mathbb{T}_{N}}
\]
where in a configuration $\eta=(\eta(x))_{x\in \T_N}\in\Sigma_{N}$,  $\eta(x)$ takes the value $0$ (resp. $1$) to indicate that the site is empty
(resp. occupied). The \emph{Facilitated Exclusion Process} (FEP) $\{\eta_{t}\}_{t\ge0}$
is a continuous-time Markov process on $\Sigma_N$ defined via the infinitesimal stochastic generator
\begin{equation}
\mathcal{L}_{N}f(\eta)=\sum_{x\in\mathbb{T}_{N}}c_{x,x+1}(\eta)\left(f\left(\eta^{x,x+1}\right)-f(\eta)\right),\label{eq:FEP-def}
\end{equation}
where $\eta^{x,y}$ denotes the configuration obtained from $\eta$
by exchanging the values at $x$ and $y$ (i.e., $\eta^{x,y}(x)=\eta(y)$,
$\eta^{x,y}(y)=\eta(x)$, and $\eta^{x,y}(z)=\eta(z)$ for $z\ne x,y$),
and
\begin{equation}
c_{x,x+1}(\eta):=\eta(x-1)\eta(x)(1-\eta(x+1))+\eta(x+2)\eta(x+1)(1-\eta(x)).\label{eq:tr-def}
\end{equation}
In other words, $c_{x,x+1}(\eta)$ encodes both the \emph{exclusion rule} (meaning particles cannot occupy the same site) as well as the FEP's \emph{kinetic constraint}: 
one needs a particle at site $x-1$ (resp. $x+2$) to \emph{facilitate} the
jump $x\to x+1$ (resp. $x+1\to x$).

Denote by $\mathbb{P}_{\mu}^{N}$ the law of the FEP trajectories
with initial distribution $\mu$ and driven by the generator $\mathcal L_N$. If the process starts from a configuration
$\eta\in\Sigma_{N}$, we simply write $\mathbb{P}_{\eta}^{N}$ instead of
$\mathbb{P}_{\delta_{\eta}}^{N}$. We denote by $|\eta|$ the number of particles of a configuration $\eta\in\Sigma_N$:
\begin{equation}
    |\eta|:= \sum_{x \in \mathbb{T}_N} \eta(x).
\end{equation}

\subsection{\label{sec2.2}Frozen, ergodic and transient components}

Because of the kinetic constraint, the FEP on the ring exhibits distinct behavior depending on the density.
As defined in \cite[Section 2B]{BES21}, a configuration $\eta \in \Sigma_N$ can either be
\begin{itemize}
\item \emph{frozen} if it has no neighboring particles, meaning $\eta_x \eta_{x+1} \equiv 0$,
\item \emph{ergodic} if it has no neighboring empty sites, meaning $(1-\eta_x) (1-\eta_{x+1}) \equiv 0$,
\item \emph{transient} otherwise. We let
\[
\Sigma_N^{\rm tr} := \left\{ \eta \in \Sigma_N : \exists x,y \in \mathbb T_N, \quad \eta(x) = \eta(x+1) = 0, \quad \eta(y) = \eta(y+1) = 1 \right\}
\]
be the set of transient configurations.
\end{itemize}

We denote by $|\eta|$  the number of particles in a configuration $\eta\in\Sigma_N$:
\begin{equation}
    |\eta| = \sum_{x \in \mathbb T_N} \eta_x.
\end{equation}
We call a configuration supercritical if $|\eta| > N/2$, subcritical if $|\eta| < N/2$, and critical if $|\eta| = N/2$. Note that by the pigeonhole principle, 
subcritical configurations cannot be ergodic, whereas supercritical configurations cannot be frozen.
%only supercritical configurations can be ergodic, and only subcritical or critical configurations can be frozen.

In what follows, given a trajectory $\{\eta(t), t\ge0 \}$ of the FEP, we denote by
\[
\mathcal T_{\rm tr}^N := \inf \left\{ t \ge 0 : \eta(t) \in \Sigma_N^{\rm tr} \right\}
\]
its \emph{transience time}, meaning the time the process needs to become either frozen or ergodic.

%Let us explain this trichotomy between the subcritical, supercritical and critical regions in more detail.
Let us consider the set of configurations with $K \in \{0,\ldots,N\}$ particles on $\mathbb{T}_N$:
\begin{equation}\label{eq:SigmaNK-def}
\Sigma_{N,K} := \left\{ \eta \in \Sigma_N : |\eta|  = K \right\} .
\end{equation}
Notice that $\Sigma_N = \bigcup_{K=0}^N \Sigma_{N,K}$. Since the FEP dynamics preserves the number of particles, any trajectory starting from a configuration in $\Sigma_{N,K}$ stays in $\Sigma_{N,K}$ for all times. Let
%Clearly, for each $\eta \in \Sigma_{N,K}$, we have $|\eta|=K$. Define
\[
\Sigma_{N,K}^{\rm tr} := \Sigma_{N,K} \cap \Sigma_N^{\rm tr}. %\qquad \text{and} \qquad \Sigma_{N,K}^{\rm rec} := \Sigma_{N,K} \setminus \Sigma_{N,K}^{\rm tr}.
\]
It is not hard to check that $\mathcal T_{\rm tr}^N$ is a.s. finite for any initial distribution \cite{BESS20}. Suppose now that the initial configuration is in $\Sigma_{N,K}$. It is not hard to check (see \cite{BESS20} for complete proofs):

\begin{itemize}
    \item \emph{Subcritical or critical case}. If $K\leq N/2$, after time $\mathcal T_{\rm tr}^N$ the system is in a \emph{frozen} configuration (and we speak of \emph{freezing time}).
    \item \emph{Supercritical case}. If $K>N/2$, after time $\mathcal T_{\rm tr}^N$ the system lives in $\Sigma_{N,K}\cap (\Sigma_N^{\rm tr})^c$ and its dynamics is irreducible inside this set.
\end{itemize}

\subsection{\label{sec2.3}Sharp bound for the subcritical freezing time}

Our first result is a sharp estimate on the transience (freezing) time in the subcritical region.
For that purpose, we start our process from the product Bernoulli measure
$\nu_{\rho}=\nu_{\rho}^{N}$ on $\Sigma_{N}$, namely
\begin{equation}
\nu_{\rho}^{N}:=\bigotimes_{x\in\mathbb{T}_{N}}\nu_{x}\qquad\text{with}\quad
\nu_x(1) = 1 - \nu_x(0) = \rho.\label{eq:nu-rho-def}
\end{equation}
In order to focus on the subcritical regime of FEP, we assume
\[
0 < \rho < \frac12,
\]
which guarantees in particular, by a standard large deviations estimate, that
\[
\limsup_{N \to \infty} \, \nu_\rho \left( |\eta| \le \frac N2 \right) = 1.
\]
%Notice that the total number of particles in the system is binomially
%distributed with $N$ trials and success probability $\rho$. In particular,
%with probability (exponentially) tending to one as $N\to\infty$, the total number
%of particles is less than or equal to $\frac{N}{2}$. Thus, with probability tending to one,
%the transience time $\mathcal{T}_{{\rm tr}}^{N}$ equals the absorption time to the frozen collection $\Sigma_{N,K}^{\rm fr}$.
Our first main result is the following, which states that the subcritical transience time is of order $\log^3N$.
\begin{thm}
\label{thm:main}For any $\rho \in (0,1/2)$, there exist constants $C_{\rm LB}=C_{\rm LB}(\rho)>0$ and $C_{\rm UB}=C_{\rm UB}(\rho)>0$ such that 
\[
\lim_{N\to\infty}\mathbb{P}_{\nu_{\rho}}^{N}\left[C_{\rm LB} \log^{3}N<\mathcal{T}_{{\rm tr}}^N<C_{\rm UB} \log^{3}N\right]=1.
\]
\end{thm}

\begin{rem}[Asymptotics of the constants near criticality, $\rho \uparrow \frac12$]
One might be interested in the asymptotics of the constants $C_{\rm LB} (\rho),C_{\rm UB} (\rho)>0$
in Theorem \ref{thm:main} as $\rho$ approaches the critical density,
i.e., as $\rho\uparrow\frac{1}{2}$.
%One might expect that the transience time $\mathcal{T}_{\rm tr}^N=\mathcal{T}_{\rm tr}^N(\rho)$ increases as $\rho$ increases, since if the system has more particles then it would naturally be harder to reach a configuration in which all particles are isolated. Moreover, the transience time at criticality $\rho = \frac12$ is expected to be polynomial in $N$ rather than polylogarithmic. Thus, it is natural to expect that both constants $C_{\rm LB}(\rho),C_{\rm UB}(\rho)$ diverge to infinity as $\rho\uparrow\frac12$.
%On the one hand, the lower bound constant $C_{\rm LB} (\rho)$ can be chosen to be independent of $\rho$ in this regime; see \eqref{eq:c6-def} and \eqref{eq:C-def}.
On the one hand, the upper bound constant $C_{\rm UB} (\rho)$ diverges to infinity as
$\rho \uparrow \frac{1}{2}$; see Remark \ref{rem:CUB-div}, \eqref{eq:lam2-def}, \eqref{eq:Lam-def},
and \eqref{eq:C'-def}. This observation is consistent with the heuristics
that the transience time of the FEP dynamics starting from product
Bernoulli $\rho$ initial distribution should increase as $\rho$
approaches the critical density $\frac{1}{2}$.
%This is indeed true for the upper bound constant $C_{\rm UB}(\rho)$;
A refined analysis which keeps track all the quantitative bounds gives
\begin{equation}
C_{\rm UB}(\rho) \simeq \left( \frac12 - \rho \right)^{-4} \log^2 \frac1{\frac12-\rho}\qquad\text{as}\quad \rho \uparrow \frac12.\label{C'(rho)}
\end{equation}

On the other hand, according to our proof technique, the lower bound constant
$C_{\rm LB}(\rho)$ does not diverge to infinity as $\rho$ approaches $\frac12$; see \eqref{eq:c6-def} and \eqref{eq:C-def}. This is due to the fact that the key part in the lower bound proof, Lemma \ref{l7}, is due to an argument which is uniform near criticality. More in detail, Lemma \ref{l7} estimates the typical number of consecutive particles in $\mathbb{T}_N$ according to the initial Bernoulli distribution $\nu_\rho^N$ (cf. \eqref{eq:nu-rho-def}), which would have a typical scale of $a(\rho)\log N$ such that $a(\rho) \uparrow a(\frac12) \in \mathbb{R}$ as $\rho \uparrow \frac12$.\footnote{In fact, $a(\rho) = (\log \frac1\rho)^{-1}$ by \cite[Example 1.5]{RAS14}.} Because of this, the constant $C_{\rm LB}(\rho)$ that we have found also converges to some real number rather than diverging to infinity.
To improve this issue and obtain $C_{\rm LB}(\rho)$ which diverges as $\rho \uparrow \frac12$, we would have to consider not only the length of consecutive particles, but also the length of the interval that those particles can travel. We did not pursue this further.
\end{rem}

\subsection{\label{sec2.4}Close-to-criticality freezing time}

Our second result is a characterization of the freezing time for the FEP close to criticality. In this regime, the CLT density fluctuations from the grand canonical (product) distributions $\nu_\rho$ do not typically allow to get arbitrarily close to criticality. For this reason, we instead introduce the canonical state $\nu_{N,K}$ as the uniform state over
$\Sigma_{N,K}$ (cf. \eqref{eq:SigmaNK-def}),
namely
\begin{equation}\label{eq:nuNK-def}
\nu_{N,K}(\eta)= {N \choose K}^{-1} {\bf 1}_{\{\eta\in \Sigma_{N,K}\}}.
\end{equation}
We want the configuration to be close to criticality, therefore we set\footnote{In this paper, $\lfloor\gamma\rfloor$ (resp. $\lceil\gamma\rceil$)
denotes the greatest (resp. least) integer less (resp. greater) than or equal to $\gamma$.}
\begin{equation}
\label{eq:Kalpha}
K=K_N^\alpha:= \left\lfloor \frac N2 - N^{1-\alpha} \right\rfloor,
\end{equation}
which sets the density in the ring at
$\rho_\alpha\simeq 1/2-N^{-\alpha}$, where $\alpha \in (0,1)$. With a slight abuse of notations, we shorten as $\Prob_{\alpha}^N$ the distribution of the FEP starting from the canonical state $\nu_{N,K_N^\alpha}$.
We claim the following.

\begin{thm}
\label{thm:main2}
Defining
\[
\ell_N=\ell_N(\alpha) :=
\begin{cases}
N & \text{for} \quad 1/2 \le \alpha < 1 , \\
N^{2\alpha} & \text{for} \quad 0<\alpha< 1/2 ,
\end{cases}
\]
for any $\varepsilon>0$,
\[
\lim_{N\to\infty}\Prob_\alpha^N \left( \ell_N^2 N^{-\varepsilon}<\mathcal{T}_{\rm tr}^N<\ell_N^2N^{\varepsilon} \right) = 1.
\]
\end{thm}

\begin{rem}\label{rem:phase-sep}
The scale $\ell_N$ defined in this theorem can be seen as the typical scale of the largest supercritical clusters. For large $\alpha$, we are very close to criticality, and the subcriticality is locally drowned in typical CLT fluctuations of the density: for this reason, one is able to find, at macroscopic scale $N$, local regions of size $\Theta(N)$ which are supercritical. In order to freeze, particles in these regions need to leave them, which happens over a typical $\ell_N^2=N^2$ diffusive timescale.

On the other hand, for small $\alpha$, any cluster of size much larger than $\ell_N=N^{2\alpha}$ will surely be subcritical because the CLT fluctuations are not enough to compensate the subcriticality of the configuration. conversely, at smaller scales than $\ell_N$,  one will be able to find supercritical clusters which need to spread out their excess particles, thus resulting in a typical timescale $\ell_N^2 = N^{4\alpha}$ for the freezing time.
\end{rem}

%\begin{rem}
%The scale $\ell_N$ is the largest scale at which supercritical clusters can be found. On the one hand, for $\alpha \ge 1/2$, we are very close to criticality, and the subcriticality is locally drowned in typical CLT fluctuations of the density; for this reason, one is able to find, at macroscopic scale $N$, local regions of size $\Theta(N)$ which are supercritical. In order to freeze, particles in these regions need to leave them, which happens over a typical $N^2$ timescale.
%
%On the other hand, suppose that $\alpha < 1/2$. Any cluster of size much larger than $\ell_N=N^{2\alpha}$ will surely be subcritical because the CLT fluctuations are not enough to compensate the subcriticality of the configuration. However, at smaller scales than $\ell_N$  one will be able to find supercritical clusters coming from the CLT fluctuations, thus resulting in a typical timescale $N^{4\alpha}$ for the freezing time.
%\end{rem}

%\cosw{SW: Please see if the next comment makes sense!}
%\coo{O: I agree that the same conclusion holds, but the proof cannot be what you indicate: we would need to keep track of how fast the probability in Theorem~\ref{thm:main} goes to one in order to argue that the change of measure does not affect the convergence.}

\begin{rem}
The same conclusion as in Theorem \ref{thm:main} holds as well if we replace the initial $\rho$-Bernoulli product measure $\nu_\rho$ with the uniform initial distribution $\nu_{N,K}$ on $\Sigma_{N,K}$ for $K = K(\rho) := \lfloor N\rho \rfloor$ where $\rho \in (0,1/2)$. Indeed, we may replace all the arguments given in Sections \ref{sec3} and \ref{sec4} for the grand-canonical product measure $\nu_\rho$ with the same results for the canonical measure $\nu_{N,K}$ via a simple comparison of measures, as e.g. in Lemma \ref{l:mupi}. This is feasible, since in principle all estimates in Section \ref{sec3} has an exponential error (although they are stated as polynomial for the reader's convenience), whereas the comparison-of-measures estimate requires only a polynomial cost in $N$.
%given that Theorem \ref{thm:main} holds for $\nu_\rho$, a simple comparison of measures, as in Lemma \ref{l:mupi}, gives immediately the desired result for $\nu_{N,K}$.
\end{rem}

\section{\label{sec3}Barriers and logarithmic supercritical intervals}

In this section, we provide a preliminary analysis on the subcritical Bernoulli initial distribution on $\mathbb{T}_N$, which will be exploited in the next section.
%in particular in the proof of the upper bound part of Theorem \ref{thm:main}.
In particular, as mentioned in the introduction, the main objective is to deduce that the typical size of a locally supercritical interval is of order $\log N$.

The lower bound part is easy, since we only need to calculate the size of a fully occupied interval. This is readily provided in the following elementary lemma, for which we provide a proof in Appendix \ref{a:proof-l7} for
the sake of completeness. The bound below is far from optimal but sufficient for our purpose.
\begin{lem}
\label{l7}There exists a constant $\Lambda' = \Lambda'(\rho)>0$ such that the following statement
holds for all sufficiently large $N$. Given i.i.d.\@ $\rho$-Bernoulli random variables $X_{1},\dots,X_{M}$
where $M \in \{ \lfloor \sqrt N \rfloor , \lceil \sqrt N \rceil \}$, with probability $1-\frac{1}N$,
there exist a failure (hole) and then at least $\Lambda'\log N+1$ consecutive
successes (particles).
\end{lem}

The upper bound part is more challenging, and constitutes the main result of this section.
To achieve this, we identify the points on the torus that the particles cannot pass through with a probability tending to $1$. In this regard, we give the following definition:

\begin{defn}
\label{def:bar}For $\eta\in\Sigma_{N}$, we say that $x\in\mathbb{T}_{N}$
is a \emph{barrier} of $\eta$ if the FEP dynamics starting from $\eta$
cannot place a particle on either $x$ or $x+1$, i.e., if 
\[
\mathbb{P}_{\eta}^{N}[\eta_{t}(x)=\eta_{t}(x+1)=0\quad \text{for all}\enspace t\ge0]=1.
\]
Denote by $\mathscr{B}(\eta)\subset\mathbb{T}_{N}$ the collection
of all barriers of $\eta$. 
\end{defn}

Note that $\mathscr{B}$ maps $\Sigma_{N}$ to the finite subsets of of $\mathbb{T}_{N}$, and that the set of barriers of a configuration is a deterministic function of this configuration.

Given a subset $A=\{x_{1},x_{2},\dots,x_{m}\}$ of $\mathbb{T}_{N}$
with $m\ge2$ where $1 \le x_{1}<x_{2}<\cdots<x_{m} \le N$, the torus $\mathbb{T}_{N}$
is divided by $A$ into $m$ intervals $\{x_{i}+1,\dots,x_{i+1}\}$
for $1\le i\le m$, where $x_{m+1} := x_{1}$. In this case, we say that
$A$ is $\gamma$-\emph{nice} if each of the $m$ intervals has cardinality
smaller than or equal to $\gamma$. It is clear that if $A\subset A'$ and $A$
is $\gamma$-nice, then $A'$ is also $\gamma$-nice.

The following lemma is the main statement of Section \ref{sec3}.
\begin{thm}
\label{thm:key} For any $\rho \in (0,\frac12)$, there exists a constant $\Lambda = \Lambda(\rho)>0$ such that
\[
\lim_{N\to\infty}\nu_{\rho}^{N}(\eta\in\Sigma_{N}:\mathscr{B}(\eta)\enspace\text{is}\enspace \Lambda \log N\text{-nice})=1.
\]
\end{thm}

To prove Theorem \ref{thm:key}, we show that on an event of probability going to $1$, we can construct a subset of $\mathbb{T}_{N}$ which is
contained in $\mathscr{B}(\eta)$ and is $\Lambda \log N$-nice. More specifically, for each $x\in \mathbb{T}_N$, we try to find a barrier $\widehat{x}\in \mathbb{T}_N$ such that $d(x,\widehat{x}) \le \Lambda \log N$. Then, collecting such $\widehat{x}$ for all $x$ would prove the theorem since $d(\widehat{x},\widehat{\widehat{x}}) \le \Lambda \log N$. The key strategy to achieve this goal is to analyze the i.i.d. $\rho$-Bernoulli random variables $\eta_{x+1},\eta_{x+2},\dots,\eta_x$ via the corresponding asymmetric random walk on $\mathbb{Z}$. This procedure will be explained in a general manner in Section \ref{sec3.1}.

The construction mentioned above relies on the following characterization of barriers.

%\coo{O: I moved the following lemma and its proof at this spot, so that it makes more sense why we introduce the divider.}

\begin{lem}
\label{l4}Instead of the torus $\mathbb{T}_{N}$, consider the FEP dynamics (cf. \eqref{eq:FEP-def}) on
the non-periodic segment $\mathbb{L}_{N}:=\{1,2,\dots,N\}$ (instead
of the torus $\mathbb{T}_{N}$), generated by
\[
\mathcal L_N' f (\eta) = \sum_{x=1}^{N-1} c_{x,x+1} (\eta) \left( f \left( \eta^{x,x+1} \right) - f(\eta) \right),
\qquad \eta \in \{0,1\}^{\mathbb L_N}, \quad f : \{0,1\}^{\mathbb L_N} \to \mathbb R,
\]
where $c_{x,x+1}(\eta)$ was defined at \eqref{eq:tr-def}.
Then for any $\eta\in\{0,1\}^{\mathbb{L}_{N}}$,
site $1$ remains empty at all times with probability $1$ in the dynamics starting from $\eta$ if and only if
\begin{equation}
2\sum_{i=1}^{\ell}\eta(i)\le\ell\qquad\text{for all}\quad1\le\ell\le N.\label{eq:l4-1}
\end{equation}
\end{lem}

\begin{proof}
Since we are in finite volume, to prove the necessary condition, it is enough to choose an arbitrary configuration $\eta$ which does not satisfy \eqref{eq:l4-1} and find a sequence of FEP jumps allowing to fill the origin. For this reason, fix $ \eta$ such that $\sum_{i=1}^{\ell^\star}\eta(i)>\ell^\star/2$ for some $1\le\ell^{\star}\le N$, {and choose the smallest such $\ell^\star$}. This means that there
are at least $\frac{\ell^{\star}+1}{2}$ particles on $\{1,\dots,\ell^{\star}\}$.
If $\eta(1)=0$, then there are at least $\frac{\ell^{\star}+1}{2}$
particles on $\{2,\dots,\ell^{\star}\}$, thus the pigeonhole principle
guarantees the existence of a pair of neighboring particles, so the
FEP mechanism can push at least one particle to the left. We are now left with a new configuration satisfying the wanted property with a strictly smaller $\ell^\star$. This cannot
be repeated infinitely many times, so there exists a strategy that puts a particle at $1$ with positive probability, proving that \eqref{eq:l4-1} is a necessary condition.

For the sufficient condition, assume that \eqref{eq:l4-1} holds for a configuration $\eta$, {meaning no segment containing the origin is more than half-occupied. Note that choosing $\ell=1$ garantees that $\eta({1})=0$}. Thus, the proof is completed if we verify
that any FEP transition from $\eta$ to $\eta^{x,x+1}$ preserves condition
\eqref{eq:l4-1}. Since a rightward jump $x\mapsto x+1$ only decreases the occupancy rate of segments containing the origin,  in this case $\eta^{x,x+1}$ will also satisfy \eqref{eq:l4-1}. We now consider the case of a leftward jump  $x\mapsto x-1$, which requires in particular that $(\eta(x-1), \eta(x), \eta(x+1))=(0,1,1)$. because the FEP dynamics is conservative, such a jump only changes the occupancy rate of the segments $\{1,\dots x-1\}$ (which increases), and $\{1,\dots x\}$, which decreases. We therefore only need to check that \eqref{eq:l4-1} still holds for $\ell=x-1$. By assumption, $\eta(x)=\eta(x+1)=1$ and 
\[\sum_{i=1}^{x+1}\eta(i)\le(x+1)/2, \]
meaning in particular  that 
\[\sum_{i=1}^{x-1}\eta(i)=\sum_{i=1}^{x-2}\eta(i)=\sum_{i=1}^{x-2}\eta^{x-1,x}(i)\leq (x-3)/2.\]
adding to both sides $\eta^{x-1,x}(x-1)=1$ finally yields as wanted that \eqref{eq:l4-1} holds for $\eta^{x-1,x}$ as well.

\end{proof}

\subsection{\label{sec3.1} {Some estimates on random walks and cut points}} {To build the tools necessary to the proof of Theorem \ref{thm:key}, we first start by proving some properties on random walks. In the following, we write $\mathbb{N}=\{1,2,3,\dots\}$ and $\mathbb{N}_{0}=\{0\}\cup\mathbb{N}$.}
%Armed with the characterization of barriers from Lemma \ref{l4}, we now aim to find, after an arbitrary site $x$ \ccl{the barrier} $\widehat x$ as mentioned above.

{Fix a sequence of i.i.d. $Ber(\rho)$} random variables
$X=(X_{n})_{n\in\mathbb{N}}$. Given an initial value $Y_0$ in $\mathbb Z$, define the associated (discrete time) asymmetric random walk $Y=(Y_{n})_{n\in\mathbb{N}_{0}}$ on $\Z$ as 
\begin{equation}
Y_{n}:= Y_0 + 2\sum_{\ell=1}^{n}X_{\ell}-n,\label{eq:Yn-def}
\end{equation}
which defines a nearest-neighbor asymmetric random walk on $\mathbb{Z}$ with probability $\rho$ (resp. $1-\rho$) to go up (resp. down). 
Note that the restriction of $Y$ to $\{1,\ldots,N\}$, given $Y_0=0$, corresponds to a height function associated with the configuration $\eta$ starting from $x$.
Denote by $P_i$ and $E_i$ the law of $Y$ starting from $Y_0 = i\in\mathbb{Z}$ and its corresponding expectation, respectively.

Define, for each $m\in\mathbb{Z}$,
\begin{equation}
\tau_{m}:=\min\{n\in\mathbb{N}_0:Y_{n}=m\},\label{eq:tau-m-def}
\end{equation}
the first hitting time of $m$.
A standard estimate implies that, for all $m \ge 0$,
\begin{equation}\label{eq:hit-prob}
P_0 ( \tau_m < \infty ) = \left(\frac{\rho}{1-\rho}\right)^m.
\end{equation}
We record a lemma regarding exponential moments of the one-step hitting time; its proof is given in Appendix \ref{a:proof-exp-moment}.

%\coo{O: I gathered the proofs using standard SRW arguments in Appendix B.}

\begin{lem}\label{l2}
If $\rho < \frac12$, for all $\lambda \ge 0$ with $4\rho(1-\rho) \le e^{-2\lambda} \le 1$,
\[
E_0 \left[ e^{\lambda \tau_{-1}} \right] = \frac{e^{-\lambda}-\sqrt{e^{-2\lambda}-4\rho(1-\rho)}}{2\rho} < \infty.
\]
\end{lem}

\begin{rem}
The exponential moment in the previous lemma diverges if $e^{-2\lambda} < 4\rho(1-\rho)$.
\end{rem}

\begin{figure}
\begin{tikzpicture}
\draw[very thick,densely dashed,-latex] (0,0)--(10,0); \draw[very thick,densely dashed,-latex] (0,-3)--(0,3);
\draw (-0.1,0) node[left]{$0$}; \draw (-0.1,1) node[left]{$n_\rho$};
\draw (0.75,0.75) node[above right]{\color{red} $Y$};

\fill[red] (0,1) circle (0.06);
\draw[thick,red] (0,1)--(0.25,0.75)--(0.5,0.5)--(0.75,0.75)--(1,0.5)--(1.25,0.25)--(1.5,0);
\draw[thick,red] (1.5,0)--(1.75,-0.25)--(2,-0.5)--(2.25,-0.25)--(2.5,-0.5)--(2.75,-0.75)--(3,-0.5)--(3.25,-0.25)--(3.5,0);
\draw[thick,red] (3.5,0)--(3.75,0.25)--(4,0)--(4.25,-0.25)--(4.5,-0.5)--(4.75,-0.75)--(5,-1)--(5.25,-0.75)--(5.5,-1);
\draw[thick,red] (5.5,-1)--(5.75,-1.25)--(6,-1.5)--(6.25,-1.25)--(6.5,-1)--(6.75,-1.25)--(7,-1)--(7.25,-1.25)--(7.5,-1.5);
\draw[thick,red] (7.5,-1.5)--(7.75,-1.75)--(8,-2)--(8.25,-2.25)--(8.5,-2)--(8.75,-2.25)--(9,-2.5)--(9.25,-2.25)--(9.5,-2.5)--(9.75,-2.75);

\draw[very thick] (1.5,-0.1)--(1.5,0.1); \draw (1.5,0.1) node[above]{$T_0 = \tau_0$};
\draw[very thick] (3.75,-0.1)--(3.75,0.1); \draw[,densely dotted] (3.75,0)--(3.75,0.25); \draw (3.75,-0.05) node[below]{$T_0'$};
\draw[very thick] (5,-0.1)--(5,0.1); \draw[,densely dotted] (5,0)--(5,-1); \draw (5,0.1) node[above]{$T_1$};
\draw[very thick] (5.25,-0.1)--(5.25,0.1); \draw[,densely dotted] (5.25,0)--(5.25,-0.75); \draw (5.25,-0.05) node[below]{$T_1'$};
\draw[very thick] (5.75,-0.1)--(5.75,0.1); \draw[,densely dotted] (5.75,0)--(5.75,-1.25); \draw (5.75,0.1) node[above]{$T_2$};
\draw[very thick] (6.5,-0.1)--(6.5,0.1); \draw[,densely dotted] (6.5,0)--(6.5,-1); \draw (6.5,-0.05) node[below]{$T_2'$};
\draw[very thick] (7.75,-0.1)--(7.75,0.1); \draw[thick,densely dotted] (7.75,0)--(7.75,-3); \draw (7.75,0.1) node[above]{$T_3 = \mathfrak{T}$};
\draw[thick,densely dotted] (0,-1.75)--(10,-1.75);
\draw[densely dotted] (0,-1)--(5,-1);
\draw[densely dotted] (0,-1.25)--(5.75,-1.25);
\draw (0,3.1) node[above]{space}; \draw (10.1,0) node[right]{time};
\end{tikzpicture}\caption{\label{fig3.1}An illustration of a sample path of $Y$ starting from $n_\rho = \lfloor c_\rho \log N \rfloor$ (cf. \eqref{nrho-crho-def}), the stopping times $\tau_0=T_0,T_0',T_1,T_1',\dots$, and the cut point $\mathfrak{T}=T_3$. One can observe that \eqref{eq:T-prop} is valid.}
\end{figure}
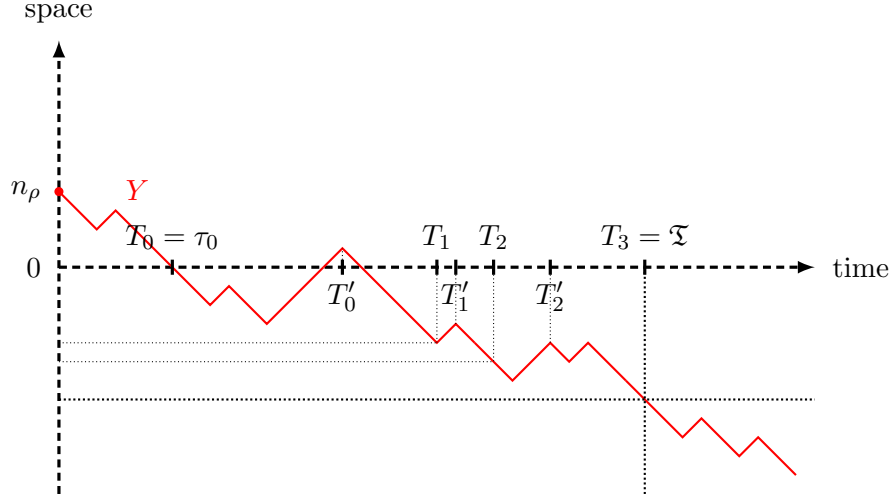

Define $T_0 = \tau_0$ and, recursively for $k\ge0$,
\begin{align}
T_k' &:= \min \{ n \ge T_k : Y_n = Y_{T_k}+1 \}, \label{eq:Tk'}\\
T_{k+1} &:= \min \left\{ n \ge T_k' : Y_n = \min_{j<n} Y_j -1 \right\}. \label{eq:Tk+1}
\end{align}
These are stopping times that take value in $\mathbb{N}_0\cup\{\infty\}$. Since $\lim_{n\to\infty}Y_n = -\infty$ a.s., given $T_k'<\infty$ we have $T_{k+1}<\infty$ almost surely. Let
\[
\mathfrak K := \min \left\{ k\ge 0: T_k'=\infty \right\} \qquad\text{and}\qquad \mathfrak{T} := T_{\mathfrak K}.
\]
Then, by \eqref{eq:Tk+1} and the fact that $T_{\mathfrak K}'=\infty$, we have
\begin{equation}\label{eq:T-prop}
\min_{n \le \mathfrak{T}} Y_n = Y_{\mathfrak{T}} = \max_{n \ge \mathfrak{T}} Y_n.
\end{equation}
From this notice that $\mathfrak{T}$ is not a stopping time since it naturally depends on the whole future. We call $\mathfrak{T}$ a \emph{cut point}. Refer to Figure \ref{fig3.1} for an illustration. 

\begin{rem}
[Interpretation of cutpoints for the FEP]
Consider a FEP on the half line, and see the random walk $Y$ as the associated height function. Starting from e.g. $Y_0=0$,  the fact that $Y_n \ge Y_{\mathfrak{T}}$ for all $n \le \mathfrak{T}$ implies that no particle on the left side of site $x:=\mathfrak{T}$ will ever cross it from left to right, while the fact that $Y_n \le Y_{\mathfrak{T}}$ for all $n \ge \mathfrak{T}$ implies that no particle on the right-hand side of  site $x:=\mathfrak{T}$ will ever cross it from right to left. In particular, the two sides of this half-infinite FEP will then evolve independently\footnote{A similar idea is used in \cite{golf} to build the golf model on $\mathbb{Z}$ for suitable initial configurations. Cut points are called separators there.}. Similar arguments can allow to build this separation in independently evolving segments in the periodic case, and will be crucial to estimate the transience time.
This heuristic will be made rigorous later in the proof of Lemma \ref{l5}, via the definition \eqref{eq:Yn-def} and the characterization in Lemma \ref{l4}.
\end{rem}
The next lemma states that $\mathfrak{T}$ has a finite exponential moment.

\begin{lem}\label{lem:exp-mom}
There exists a constant $\lambda_1=\lambda_1(\rho)>0$ such that
\[
E_0 \left[ e^{\lambda_1 \mathfrak{T}} \right] < \infty.
\]
\end{lem}

\begin{proof}
The strong Markov property and \eqref{eq:hit-prob} imply that, under $P_0$, the random variable $\mathfrak K$ is geometrically distributed with parameter
\[
P_0 \left( T_0' < \infty \right) = \frac{\rho}{1-\rho} < 1.
\]
In turn, for $\lambda>0$, we can compute 
\begin{align*}
E_0 \left[ e^{\lambda \mathfrak{T}} \right] = \sum_{k=0}^\infty E_0 \left[ e^{\lambda T_k} {\bf 1}_{\mathfrak K=k} \right] & = \sum_{k=0}^\infty E_0 \left[ {\bf 1}_{T_k'=\infty} \prod_{j=0}^{k-1} e^{\lambda(T_{j+1}-T_j)} {\bf 1}_{T_j' < \infty} \right] \\
& = \sum_{k=0}^\infty \left( 1 - \frac{\rho}{1-\rho} \right) E_0 \left[ e^{\lambda T_1} {\bf 1}_{T_0' < \infty} \right]^k ,
\end{align*}
where the last line follows from the strong Markov property applied recursively at stopping times $T_k, T_{k-1}, \dots, T_0$. Thus, it suffices to find $\lambda>0$ such that
\[
E_0 \left[ e^{\lambda T_1} {\bf 1}_{T_0' < \infty} \right] < 1.
\]
Notice that $T_0'$, if finite, is odd. Starting from $Y_0=0$, since $Y_{T_0'}=1$, up until $T_0'$ exactly half the steps are spent going down, therefore given $T_0' < \infty$ we have
\[
\min_{0 \le n \le T_0'}Y_n \ge - \frac{T_0'-1}{2}, \qquad \text{thus} \qquad Y_{T_1} - Y_{T_0'} \ge - \frac{T_0'}{2} -\frac32.
\]
Thus, via the strong Markov property, we may calculate
\begin{align*}
E_0 \left[ e^{\lambda T_1} {\bf 1}_{T_0' < \infty} \right] & \le \sum_{k=0}^\infty E_0 \left[ e^{\lambda (2k+1)} {\bf 1}_{T_0'=2k+1} E_0 \left[ e^{\lambda \left( T_1 - T_0' \right) } \Big| \mathcal{F}_{T_0'} \right] \right] \\
& \le \sum_{k=0}^\infty E_0 \left[ e^{\lambda (2k+1)} {\bf 1}_{T_0'=2k+1} \right] E_0 \left[ e^{\lambda \tau_{-k-2}} \right] \\
& = \sum_{k=0}^\infty e^{\lambda(2k+1)} P_0 \left( T_0'=2k+1 \right) E_0 \left[ e^{\lambda \tau_{-1}} \right]^{k+2}.
\end{align*}
A combinatorial argument implies that $P_0(T_0'=2k+1) = C_k \rho^{k+1}(1-\rho)^k$ where $C_k$ is the $k$-th Catalan number. Thus,
\begin{align*}
E_0 \left[ e^{\lambda T_1} {\bf 1}_{T_0' < \infty} \right] & \le \sum_{k=0}^\infty e^{\lambda(2k+1)} E_0 \left[ e^{\lambda \tau_{-1}} \right]^{k+2} C_k \rho^{k+1} (1-\rho)^k \\
& = \rho e^\lambda E_0 \left[ e^{\lambda \tau_{-1}} \right]^2 \sum_{k=0}^\infty C_k \left( \rho (1-\rho) e^{2\lambda} E_0 \left[ e^{\lambda \tau_{-1}} \right] \right)^k.
\end{align*}
If $\lambda>0$ satisfies 
\begin{equation}\label{eq:2nd-cond}
g_\rho (\lambda) := \rho (1-\rho) e^{2\lambda} E_0 \left[ e^{\lambda \tau_{-1}} \right] < \frac14,
\end{equation}
via the generating function for Catalan numbers, we obtain that
\begin{align*}
E_0 \left[ e^{\lambda T_1} {\bf 1}_{T_0' < \infty} \right] & \le \rho e^\lambda E_0 \left[ e^{\lambda \tau_{-1}} \right]^2 \frac{1-\sqrt{1-4 \rho (1-\rho) e^{2\lambda}E_0 [ e^{\lambda \tau_{-1}} ]}}{2 \rho (1-\rho) e^{2\lambda} E_0 [ e^{\lambda \tau_{-1}} ] } <  \frac{E_0 \left[ e^{\lambda \tau_{-1}} \right]}{2 (1-\rho) e^{\lambda} } < 1,
\end{align*}
where the last inequality follows from Lemma \ref{l2}. Therefore, it suffices to find $\lambda_1>0$ that satisfies \eqref{eq:2nd-cond}.
Taking $\lambda = 0$ in the left-hand side of \eqref{eq:2nd-cond} gives $g_\rho(0) = \rho(1-\rho) < \frac14$.
Thus, there exists such constant $\lambda_1$ since $g_\rho$ is continuous.
\end{proof}

\begin{rem}\label{rem:CUB-div}
We remark that $\lambda_1(\rho) \downarrow 0$ as $\rho \uparrow \frac12$, Indeed, $\lambda_1$ satisfies \eqref{eq:2nd-cond} thus $E_0[e^{\lambda_1 \tau_{-1}}] < \infty$, which then implies via Lemma \ref{l2} that $e^{-2\lambda_1} \ge 4 \rho (1-\rho)$.
\end{rem}

The following immediate corollary will be used in the sequel. Define
\begin{equation}\label{nrho-crho-def}
n_\rho := \lfloor c_\rho \log N \rfloor, \qquad \text{where} \qquad c_\rho := \frac2{\log \frac{1-\rho}{\rho}} .
\end{equation}
The constant $n_\rho$ serves as a starting point of $Y$, which is introduced to overkill the particles which may potentially come over from the left side of $x$ (i.e., from $x-1,x-2,\dots$).

\begin{cor}\label{c0}
There exist constants $\lambda_2 = \lambda_2(\rho), \lambda_3 = \lambda_3(\rho) > 0$ such that
\[
P_{n_\rho} ( \mathfrak{T} > \lambda_2 \log N ) \le \frac{\lambda_3}{N^2}.
\]
\end{cor}

\begin{proof}
By Lemma \ref{lem:exp-mom}, the Markov inequality, and the strong Markov property at $\tau_0$, for $\lambda_1>0$ which was introduced in Lemma \ref{lem:exp-mom},
\[
P_{n_\rho} ( \mathfrak{T} > \lambda_2 \log N ) \le e^{-\lambda_1 \lambda_2 \log N} E_{n_\rho} \left[ e^{\lambda_1 \mathfrak{T}} \right] = N^{-\lambda_1 \lambda_2} E_{n_\rho} \left[ e^{\lambda_1 \tau_0} \right] E_0  \left[ e^{\lambda_1 \mathfrak{T}} \right] .
\]
Since $\lambda_1$ satisfies \eqref{eq:2nd-cond}, we have $4\rho(1-\rho) < e^{-2\lambda_1}$, thus by Lemma \ref{l2} and the strong Markov property applied $n_\rho$ times,
\[
E_{n_\rho} \left[ e^{\lambda_1 \tau_0} \right] = E_0 \left[ e^{\lambda_1 \tau_{-1}} \right]^{n_\rho} \le  N^{c_\rho \log E_0 \left[ e^{\lambda_1 \tau_{-1}} \right] } .
\]
Thus, by choosing $\lambda_3 = E_0 [e^{\lambda_1 \mathfrak{T}}] \in (0,\infty)$ and $\lambda_2 >0$ such that
\begin{equation}\label{eq:lam2-def}
-\lambda_1 \lambda_2 + c_\rho \log E_0 \left[ e^{\lambda_1 \tau_{-1}} \right] \le -2,
\end{equation}
the inequality holds as desired.
\end{proof}

Before moving on, we record another corollary which is necessary in the
next subsection. Cramér's theorem implies that
\begin{equation}\label{eq:LDP}
P(Y_n-Y_0\ge-\gamma n) \le e^{-f_{\rho}(\gamma)n} \qquad \text{for all} \quad \gamma\in (-1,1-2\rho), \quad n\in\mathbb{N}, 
\end{equation}
where
\[
f_\rho(\gamma) = \frac{1-\gamma}2 \log \frac{1-\gamma}{2\rho} + \frac{1+\gamma}2 \log \frac{1+\gamma}{2(1-\rho)}.
\]

\begin{cor}
\label{c1} For all sufficiently large $N$,
\[
P_{n_\rho}(Y_{N}+c_{\rho}\log N>Y_{\mathfrak{T}})\le\frac{1+\lambda_3}{N^{2}}.
\]
\end{cor}

\begin{proof}
By Corollary \ref{c0}, $\mathfrak{T} \le \lambda_2 \log N$
except with probability at most $\frac{\lambda_3}{N^{2}}$. Thus, for all sufficiently large $N$ (such that $c_\rho\log N\le\frac{1-2\rho}2(N-\lambda_2\log N)$),
\begin{align*}
P_{n_\rho} (Y_{N}+c_{\rho}\log N>Y_{\mathfrak{T}}) & \le \sum_{k=0}^{\lfloor \lambda_2 \log N \rfloor} P (Y_N - Y_k > - c_\rho \log N )+\frac{\lambda_3}{N^{2}} \\
& \le \sum_{k=0}^{\lfloor \lambda_2 \log N \rfloor}  P\left( Y_{N-k }-Y_0 > -\frac{1-2\rho}{2}(N-k)\right) + \frac{\lambda_3}{N^{2}}.
\end{align*}
Applying \eqref{eq:LDP} with $\gamma=\frac{1-2\rho}{2}$ (for which $f(\gamma)>0$) concludes the proof.
\end{proof}

\subsection{\label{sec3.2}Proof of Theorem \ref{thm:key}}

Recall from \eqref{eq:nu-rho-def} that $\nu_{\rho}^{N}=\otimes_{x\in\mathbb{T}_{N}}\nu_{x}$
with $\nu_{x}(0)=1-\rho$ and $\nu_{x}(1)=\rho$.
For any $x\in\mathbb{T}_{N}$, consider
the following sequence of random variables
\begin{equation}
\eta_{x+1},\eta_{x+2},\dots,\eta_{x},X_{N+1},X_{N+2},\dots\label{eq:iid-x}
\end{equation}
where the first $N$ random variables are the site marginals of $\eta \sim \nu_\rho^N$ encountered by crossing the full ring from
$x+1\in\mathbb{T}_{N}$ to $x\in\mathbb{T}_{N}$, and the remaining $X_{i}$, $i\ge N+1$ are i.i.d. $\rho$-Bernoulli
random variables that are independent of $\eta$. Then,
following the construction given in Section \ref{sec3.1}
with respect to \eqref{eq:iid-x}, under the initial condition $Y_0 = n_{\rho}$ (cf. \eqref{nrho-crho-def}), denote by $\mathfrak{T}_{x}$ the associated
cut point $\mathfrak T$ and on the event $\{\mathfrak{T}_{x}\leq N\}$ define the (random) site
\[
\widehat{x}:=x+\mathfrak{T}_{x}\in\mathbb{T}_{N}.
\]

\begin{lem}
\label{l5}There exists a constant $\lambda_4 = \lambda_4(\rho)>0$ such that for each
$x\in\mathbb{T}_{N}$,
\[
\nu_{\rho}^{N}(\eta\in\Sigma_{N}:\widehat{x}\in\mathscr{B}(\eta))\ge1-\frac{\lambda_4}{N^{2}}.
\]
\end{lem}

\begin{proof}
Fix $\eta\in\Sigma_{N}$ and $x\in\mathbb{T}_{N}$. Consider the one-dimensional
line
\[
\mathbb L_N^x := \{\widehat{x}+1,\widehat{x}+2,\dots,\widehat{x}\}
\]
obtained
by disconnecting the edge between $\widehat{x}$ and $\widehat{x}+1$
in $\mathbb{T}_{N}$.
By the canonical coupling between two FEP dynamics on $\mathbb T_N$ and $\mathbb L_N^x$,
the event that $\widehat{x}$ is a barrier of
$\eta$ follows from the event that FEP on this line
starting from $\eta$ cannot place a particle on either $\widehat{x}+1$
or $\widehat{x}$ at all times. By Lemma \ref{l4}, the latter event is equivalent
to
\[
2\sum_{i=1}^{\ell}\eta_{\widehat{x}+i}\le\ell\qquad\text{and}\qquad2\sum_{i=1}^{\ell}\eta_{\widehat{x}+1-i}\le\ell\qquad\text{for all}\quad1\le\ell\le N.
\]
In the notation of the process $Y=(Y_{\ell})_{\ell\in\mathbb{N}_{0}}$
defined in \eqref{eq:Yn-def}, with $Y_0 := n_\rho$, this is equivalent to:
\begin{itemize}
\item $Y_{\mathfrak{T}_{x}+\ell}-Y_{\mathfrak{T}_{x}}\le0$ for all $1\le\ell\le N-\mathfrak{T}_{x}$;
\item $(Y_{N}-Y_{\mathfrak{T}_{x}})+(Y_{\ell}-Y_{0})\le0$ for all $1\le\ell\le\mathfrak{T}_{x}$;
\item $Y_{\mathfrak{T}_{x}}-Y_{\mathfrak{T}_{x}-\ell}\le0$ for all $1\le\ell\le\mathfrak{T}_{x}$;
\item $(Y_{\mathfrak{T}_{x}}-Y_{0})+(Y_{N}-Y_{N-\ell})\le0$ for all $1\le\ell\le N-\mathfrak{T}_{x}$.
\end{itemize}
The first and third conditions follow from the fact that $\mathfrak{T}_{x}$
is the cut point (see \eqref{eq:T-prop}):
\[
\begin{cases}
Y_{\mathfrak{T}_{x}+\ell}\le Y_{\mathfrak{T}_{x}} & \text{for all}\quad1\le\ell\le N-\mathfrak{T}_{x},\\
Y_{\mathfrak{T}_{x}-\ell}\ge Y_{\mathfrak{T}_{x}} & \text{for all}\quad1\le\ell\le\mathfrak{T}_{x}.
\end{cases}
\]
The second condition holds with a probability tending to $1$ due to Corollary \ref{c1} and \eqref{eq:hit-prob}:
\[
\nu_{\rho}^{N}(Y_{N}-Y_{\mathfrak{T}_{x}}\le-c_{\rho}\log N)\ge1-\frac{1+\lambda_3}{N^{2}},
\]
and
\begin{equation}
\nu_{\rho}^{N}(Y_{\ell}-Y_{0}\le c_{\rho}\log N\enspace\text{for all}\enspace \ell\ge 1 )\ge1-\left(\frac{\rho}{1-\rho}\right)^{c_{\rho}\log N} = 1-\frac{1}{N^{2}}.\label{eq:l5-1}
\end{equation}
The fourth condition holds with a probability tending to $1$ due to \eqref{eq:hit-prob} and $Y_{\mathfrak{T}_{x}}\le0$:
\begin{align*}
\nu_{\rho}^{N}(Y_{N}-Y_{N-\ell} & \le Y_{0}\enspace\text{for all}\enspace1\le\ell\le N-\mathfrak{T}_{x})\\
 & \ge\nu_{\rho}^{N}(Y_{\ell}-Y_{0}\le c_{\rho}\log N\enspace\text{for all}\enspace\ell\ge1)\ge1-\frac{1}{N^{2}}.
\end{align*}
Therefore, we may lower bound the probability that $\widehat{x}$
is a barrier with
\[
\nu_{\rho}^{N}(\eta\in\Sigma_{N}:\widehat{x}\in\mathscr{B}(\eta))\ge1-\frac{3+\lambda_3}{N^{2}},
\]
which proves the lemma with $\lambda_4:=3+\lambda_3$.
\end{proof}
Now, we are ready to prove Theorem \ref{thm:key}.
\begin{proof}[Proof of Theorem \ref{thm:key}]
 Applying Lemma \ref{l5} for all $x\in\mathbb{T}_{N}$,
\[
\nu_{\rho}^{N}(\eta:\widehat{x}\in\mathscr{B}(\eta)\enspace\text{for all}\enspace x\in\mathbb{T}_{N})\ge1-\frac{\lambda_4}{N}.
\]
In addition, applying Corollary \ref{c0} for each $x$,
\[
\nu_{\rho}^{N}(\mathfrak{T}_{x}\le \lambda_2 \log N\enspace\text{for all}\enspace x\in\mathbb{T}_{N})\ge1-\frac{\lambda_3}{N}.
\]
Thus, in the intersection of these two events, $\{\widehat{x}:x\in\mathbb{T}_{N}\}$
is a subset of $\mathscr{B}(\eta)$ and is $\lambda_2 \log N$-nice, since
for any $\widehat{x}$,
\[
d\,\left(\widehat{x},\widehat{\widehat{x}}\right)=\mathfrak{T}_{\widehat{x}}\le \lambda_2 \log N.
\]
Therefore,
\[
\nu_{\rho}^{N}(\eta\in\Sigma_{N}:\mathscr{B}(\eta)\enspace\text{is}\enspace \lambda_2 \log N\text{-nice}) \ge 1-\frac{\lambda_4+\lambda_3}{N}\xrightarrow{N\to\infty}1,
\]
which concludes the proof of Theorem \ref{thm:key} with
\begin{equation}\label{eq:Lam-def}
\Lambda (\rho) := \lambda_2 (\rho).
\end{equation}
\end{proof}

\section{\label{sec4}Proof of Theorem \ref{thm:main}}

\subsection{Mapping to the SSEP with Traps (SWT)}
\label{sec:SWTFEPshort}
{Following \cite[Section 1.1]{EM24}, we now describe formally the mapping from the SSEP to the SWT, illustrated in Fig. \ref{fig4.1}, and defined in details in Appendix \ref{secA}, which will be important in what follows.

\medskip

\emph{Static mapping:} Any FEP configuration $\eta$ on $\T_N$ with $K$ particles can be associated with a SWT configuration $\xi=(\xi(i))_{i\in\T_K}\in \Z_{\leq 1}^{\T_K},$ where each site $i$ of the SWT is either occupied by a particle ($\xi(i)=1$), or contains a trap of depth $k\geq 0$ ($\xi(i)=-k$), with a trap of depth $0$ being interpreted as a classical SSEP empty site. This correspondance is defined by keeping track of the positions $X_1<\dots < X_K< X_{K+1}:=X_1$ of the particles in $\eta$, and setting 
\begin{equation}
\label{eq:DEFSWTshort}
\xi(i)=2-(X_{i+1}-X_i),
\end{equation}
meaning that site $i$ in $\xi$ is occupied by a particle if particles $i$ and $i+1$ are neighbors in $\eta$, and it is a trap of depth $k\geq 0$ if $k+1$ empty sites separate particles $i$ and $i+1$ in $\eta$. Note that this SWT configuration must satisfy 
\begin{equation}
\label{eq:mappedGammaNK}
\sum_{i\in \T_K} \xi(i)=2K-N.
\end{equation}

\medskip

\emph{Dynamic mapping:} This mapping can now follow a FEP trajectory $\{\eta_t,\; t\geq 0\}$, by keeping track of the trajectories $X_1(t),\dots X_K(t), X_{K+1}(t):=X_1(t)$ of the FEP's particles, keeping in mind that order is preserved because the FEP particles cannot cross eachother. Doing so and applying \eqref{eq:DEFSWTshort} at time $t$ defines the mapped SWT trajectory $\{\xi_t,\; t\geq 0\}$, which constructs a Markov process in which
\begin{itemize}
\item Particles jump at rate one to any neighboring site which is not already occupied by a particle. 
\item When a particle jumps into a trap, the particle is destroyed, and the trap loses one depth.
\end{itemize}
Clearly, the SWT displays the same active/inactive transition than the FEP: initially, particles and traps can coexist (transient period), however after the transience time, either only particles or only traps subsist depending on the initial sign of $\sum_{i\in \T_K} \xi(i)$.
}

\subsection{\label{sec4.1}Upper bound}

{With this mapping, we are now in a position to prove the upper bound} of Theorem \ref{thm:main}.
Namely, our objective is to verify that
\begin{equation}
\lim_{N\to\infty}\mathbb{P}_{\nu_{\rho}}^{N}\left[\mathcal{T}_{{\rm tr}}^N<C_{\rm UB}\log^{3}N\right]=1.\label{eq:main-UB}
\end{equation}
{For shortness, we simply write} $\mathcal{T}_{\rm tr} := \mathcal{T}_{\rm tr}^N$. By Theorem \ref{thm:key} in the previous section, with probability tending to one as $N\to\infty$, we may start with
an initial configuration $\eta\in\Sigma_{N}$ such that the barrier
set $\mathscr{B}(\eta)$ is $\Lambda \log N$-nice. Write $\mathscr{B}(\eta)=\{x_{1},x_{2},\dots,x_{\mathfrak{m}(\eta)}\}$
in  increasing order, and define
\begin{equation}
\mathbb{I}_{i}:=\{x_{i}+1,x_{i}+2,\dots,x_{i+1}\}\qquad\text{for}\quad1\le i\le\mathfrak{m}(\eta).\label{eq:Ii-def}
\end{equation}
{where we identified} $x_{\mathfrak{m}(\eta)+1} := x_{1}$. Then, $\mathbb{T}_{N}=\mathbb{I}_{1}\sqcup\cdots\sqcup\mathbb{I}_{\mathfrak{m}(\eta)}$
is a disjoint partition and $|\mathbb{I}_{i}| \le \Lambda \log N$ for all $1\le i\le \mathfrak{m}(\eta)$. Since
$\mathscr{B}(\eta)$ is the {set of barriers}, starting from $\eta$, the
dynamics leaves $x_{i}$ and $x_{i}+1$ empty for all $1\le i\le\mathfrak{m}(\eta)$.
This means that if we denote
by $\mathcal{T}_{{\rm tr}}^{{(i)}}$ the transience/freezing time of the initial configuration $\eta$ restricted to $\mathbb{I}_{i}$, we have
\begin{equation}
{\mathcal{T}_{{\rm tr}}=\max\{\mathcal{T}_{{\rm tr}}^{(i)}\, :\, 1\leq i\leq  \mathfrak{m}(\eta)\}},\label{eq:trans-dec}
\end{equation}
where, by construction of the barriers, the $\mathcal{T}_{\rm tr}^{{(i)}}$, $1 \le i \le \mathfrak{m}(\eta)$ are independent, so that
\begin{equation}
\mathbb{P}_\eta^N \left[ \mathcal{T}_{\rm tr} < C_{\rm UB} \log^3 N \right] = \prod_{i=1}^{\mathfrak{m}(\eta)}
\mathbb{P}_\eta^N \left[ \mathcal{T}_{\rm tr}^{{(i)}} < C_{\rm UB} \log^3 N \right] .
\end{equation}

Exploiting the mapping between the FEP and the {SWT} (described briefly in the previous paragraph \ref{sec:SWTFEPshort}, and in details in Appendix \ref{secA}), we may rewrite {the right-hand side of the last identity as}
\[
\prod_{i=1}^{\mathfrak{m}(\eta)}
\mathbb{P}_\eta^N \left[ \mathcal{T}_{\rm tr}^{{(i)}} < C_{\rm UB} \log^3 N \right] =
\prod_{i=1}^{\mathfrak{m}(\eta)}
\mathbb{Q}_{\xi} \left[ \mathcal{T}_{\rm tr}^{{(i)}} < C_{\rm UB} \log^3 N \right] ,
\]
{where $\mathbb{Q}_{\xi}$ denotes the law of the SWT dynamics starting from the SWT configuration  $\xi$ (which can be defined as $\Psi\circ\Upsilon \circ \Phi^{-1}(\eta)$,
%\coo{O: I changed the letters to match \eqref{eq:corres}, but it remains to specify a labeling so as to make $\Phi^{-1}$ well-defined.}
according to  \eqref{eq:corres}, Notation \ref{not:inv-image}, and \eqref{eq:trans-corr}).

 By slight abuse of notation, the transience time  $\mathcal{T}_{\rm tr}^{{(i)}}$} on the right-hand side is the transience time of the SWT dynamics started from $\xi$ restricted to the interval mapped from the particles in $\mathbb{I}_i$. 
 Let us denote this interval as $\mathbb{J}_i \subset \mathbb{T}_K$, where $K=|\eta|=\sum_{x \in \mathbb{T}_N} \eta(x)$. Note that $|\mathbb{J}_i| =\sum_{i\in\mathbb{I}_i}\eta(x) \le |\mathbb{I}_i| \le \Lambda \log N$.

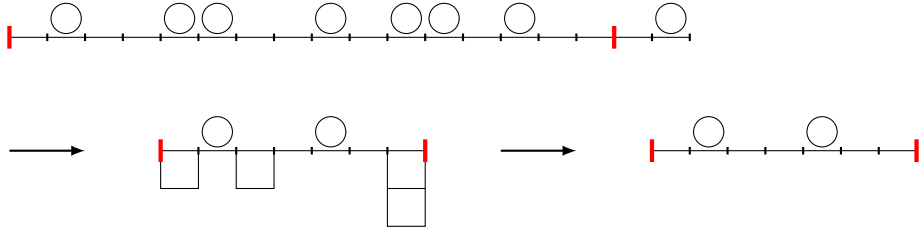
\begin{figure}
\begin{tikzpicture}
\begin{scope}[scale=0.5]
\draw (0,0)--(18,0);
\foreach \i in {0,...,18} { \draw[thick] (\i,-0.1)--(\i,0.1); }
\foreach \i in {1,4,5,8,10,11,13,17} { \draw (\i+0.5,0.5) circle (0.4); }
\draw[ultra thick,red] (0,-0.3)--(0,0.3); \draw[ultra thick,red] (16,-0.3)--(16,0.3);
\end{scope}

\draw[thick,-latex] (0,-1.5)--(1,-1.5);

\begin{scope}[shift={(2,-1.5)},scale=0.5]
\draw (0,0)--(7,0);
\foreach \i in {0,...,7} { \draw[thick] (\i,-0.1)--(\i,0.1); }
\foreach \i in {1,4} { \draw (\i+0.5,0.5) circle (0.4); }
\foreach \j in {0} { \draw[] (0,-\j)--(0,-\j-1)--(1,-\j-1)--(1,-\j); }
\foreach \j in {0} { \draw[] (2,-\j)--(2,-\j-1)--(3,-\j-1)--(3,-\j); }
\foreach \j in {0,1} { \draw[] (6,-\j)--(6,-\j-1)--(7,-\j-1)--(7,-\j); }
\draw[ultra thick,red] (0,-0.3)--(0,0.3); \draw[ultra thick,red] (7,-0.3)--(7,0.3);
\end{scope}

\draw[thick,-latex] (6.5,-1.5)--(7.5,-1.5);

\begin{scope}[shift={(8.5,-1.5)},scale=0.5]
\draw (0,0)--(7,0);
\foreach \i in {0,...,7} { \draw[thick] (\i,-0.1)--(\i,0.1); }
\foreach \i in {1,4} { \draw (\i+0.5,0.5) circle (0.4); }
\draw[ultra thick,red] (0,-0.3)--(0,0.3); \draw[ultra thick,red] (7,-0.3)--(7,0.3);
\end{scope}
\end{tikzpicture}
\caption{\label{fig4.1}A FEP configuration $\eta$ on $\mathbb{I}_i \subset \mathbb{T}_N$ (delimited by thick vertical dashes), and the left-most part of $\mathbb{I}_{i+1}$. Below, its corresponding SWT configuration $\xi$ on the mapped interval $\mathbb{J}_i \subset \mathbb{T}_K$, and the configuration $\xi_i''$ obtained from $\xi$ by filling up all the traps in $\mathbb{J}_i$.}
\end{figure}

Now fix $1 \le i \le \mathfrak{m}(\eta)$. Since the SWT particles on $\mathbb{J}_i$ cannot escape the interval, 
if we define $\xi_i' \in \Gamma_{N,K}$ (cf. \eqref{eq:SWT-space}) as
\[
\xi_i' (x):=\begin{cases}
\xi(x) & \text{if} \quad x \in \mathbb{J}_i,\\
-N & \text{if} \quad x \not\in \mathbb{J}_i,
\end{cases}
\]
then the two SWT trajectories starting from $\xi$ and $\xi'$
are the same inside $\mathbb{J}_i$, thus
\[
\mathbb{Q}_{\xi}\left[\mathcal{T}_{\rm tr}^{{(i)}}<C_{\rm UB}\log^{3}N\right]=\mathbb{Q}_{\xi_i'}\left[\mathcal{T}_{\rm tr}^{{(i)}}<C_{\rm UB}\log^{3}N\right].
\]

Next, due to the monotonicity property of the SWT (cf. \cite[Section 1.4]{EM24}),
the transience time starting from $\xi_i'$ is upper bounded
by the transience time starting from a new configuration $\xi_i'' \ge \xi_i'$, {(meaning the order holds at every site)}
which is defined as
\[
\xi_{i}''(x):=\begin{cases}
1 & \text{if} \quad \xi_i'(\eta)(x)=1 \quad \text{and} \quad x \in \mathbb{J}_i,\\
0 & \text{if}\quad \xi_i'(\eta)(x)\le0\quad\text{and}\quad x \in \mathbb{J}_i,\\
-N & \text{if}\quad x \not\in \mathbb{J}_i .
\end{cases}
\]
Namely, $\xi_{i}''$ is obtained from $\xi_i'$ by removing
all traps in $\mathbb{J}_i$. Thus,
\[
\mathbb{Q}_{\xi_i'}\left[\mathcal{T}_{\rm tr}^{{(i)}}<C_{\rm UB}\log^{3}N\right]\ge\mathbb{Q}_{\xi_{i}''}\left[\mathcal{T}_{\rm tr}^{{(i)}}<C_{\rm UB}\log^{3}N\right].
\]
Note that starting from $\xi_{i}''$, $\mathcal{T}_{\rm tr}^{{(i)}}$
is exactly the time that all particles on $\mathbb{J}_i$ are absorbed
at either $x_i$ or $x_{i+1}+1$. See Figure \ref{fig4.1} for an example of this construction from $\eta \in \Sigma_{N,K}$ to $\xi_i'' \in \Gamma_{N,K}$.

Finally, we may interpret the SWT dynamics on $\mathbb{J}_i$ starting from $\xi_{i}''$
as labeled simple random walks on $\mathbb{J}_{i}$
with rate $1$ to both left and right which gets absorbed at $x_i$
or $x_{i+1}+1$.\footnote{The idea is to assign independent clocks to the edges and, every time
a clock rings, the states of the two vertices on that edge are exchanged.
Clearly, these random walks are not independent of each other. Nevertheless,
we only use the fact that the trajectory of each single particle follows
the simple random walk on $\mathbb{J}_{i}$.} By Lemma \ref{lem:SRW}-(1) with $M = |\mathbb J_i|$,
\[
\mathbb{Q}_{\xi_{i}''}\left[\mathcal{T}_{\rm tr}^{{(i)}}\ge C_{\rm UB}\log^{3}N\right]\le |\mathbb{J}_i| \times2\exp\left(-\frac{2\pi C_{\rm UB}\log^{3}N}{9 |\mathbb{J}_i|^{2}}\right)\le2 \Lambda \log N\times N^{-\frac{2\pi C_{\rm UB}}{9\Lambda^{2}}}.
\]
Combining all the displayed arguments, along with the fact that $\mathfrak{m}(\eta)\le N$,
\begin{equation}
\mathbb{P}_{\eta}^{N}\left[\mathcal{T}_{{\rm tr}}<C_{\rm UB}\log^{3}N\right]\ge\left(1-2\Lambda \log N\times N^{-\frac{2\pi C_{\rm UB}}{9\Lambda^{2}}}\right)^{N}.\label{eq:UB-bound}
\end{equation}
Therefore, by taking $C_{\rm UB}>0$ such that 
\begin{equation}
2\pi C_{\rm UB}>9\Lambda^2,\label{eq:C'-def}
\end{equation}
the right-hand side in \eqref{eq:UB-bound} converges to $1$ and
we have proved \eqref{eq:main-UB}.

\subsection{\label{sec4.2}Lower bound}

In this subsection, we prove the lower bound part of Theorem \ref{thm:main};
namely, we prove
\begin{equation}
\lim_{N\to\infty}\mathbb{P}_{\nu_{\rho}}^{N}\left[\mathcal{T}_{{\rm tr}}^N>C_{\rm LB}\log^{3}N\right]=1.\label{eq:main-LB}
\end{equation}
As we sample the initial configuration according to $\nu_\rho$, with probability tending to one we start from a configuration $\eta$ which has
at least one particle and one empty site. As done in Section \ref{sec4.1}, {following Section \ref{sec:SWTFEPshort}}
we identify the FEP dynamics with its corresponding SWT dynamics starting from $\xi := (\Psi \circ \Upsilon \circ \Phi^{-1})(\eta)$ (cf. Notation \ref{not:inv-image}),
%\coo{O: same remark as above about $\Phi^{-1}$}
such that
\[
\mathbb{P}_{\eta}^{N}\left[\mathcal{T}_{{\rm tr}}>C_{\rm LB}\log^{3}N\right]=\mathbb{Q}_{\xi}\left[\mathcal{T}_{{\rm tr}}>C_{\rm LB}\log^{3}N\right],
\]
where once again $\mathbb{Q}_{\xi}$ denotes the law of the SWT trajectories
starting from $\xi$ and we abbreviated as $\mathcal{T}_{\rm tr} := \mathcal{T}_{\rm tr}^N$.

Divide $\mathbb{T}_{N}$ into $b_{N}\in\{\lfloor\sqrt{N}\rfloor,\lceil\sqrt{N}\rceil\}$
intervals of length $\lfloor\sqrt{N}\rfloor$ or $\lceil\sqrt{N}\rceil$.
Applying Lemma \ref{l7} to each interval, we can find $b_N$ non-adjacent
%\coo{$b_N$ disjoint (and therefore $b'_N:=\lfloor b_{N}/2\rfloor$ non-adjacent)}
%\cosw{SW: we don't need $b_N'$ since the initial failure in Lemma \ref{l7} already guarantees that $\mathbb J_i$ are non-adjacent.}
sequences of at least $\Lambda' \log N +1$ consecutive particles
with probability
\begin{equation}\label{e:proba-consecutive-part}
1-\frac{b_N}N \ge1-\frac{2}{\sqrt N},
\end{equation}
for sufficiently large $N$. Then, these non-adjacent sequences of
at least $\Lambda' \log N+1$ consecutive particles in the
FEP system convert into non-adjacent sequences of at least $\Lambda' \log N$
consecutive particles in the corresponding SWT system. Denote by $\mathbb{J}_{1},\dots,\mathbb{J}_{b_{N}}$
these non-adjacent intervals inside $\mathbb{T}_K$ (where $K=|\eta|$) such that $|\mathbb{J}_{i}|\ge \Lambda' \log N$.
Clearly, $\xi(x)=1$ for all $x\in\mathbb{J}_{i}$.

Now, appealing again to the monotonicity property of the SWT (cf.
\cite[Section 1.4]{EM24}),
\[
\mathbb{Q}_{\xi}\left[\mathcal{T}_{{\rm tr}}>C_{\rm LB}\log^{3}N\right]\ge\mathbb{Q}_{\xi}\left[\mathcal{T}_{{\rm tr}}>C_{\rm LB}\log^{3}N\right],
\]
where $\xi'\le\xi$ is obtained from $\xi$ by declaring every state
on $\mathbb{T}_{K}\setminus(\mathbb{J}_{1}\cup\cdots\cup \mathbb{J}_{b'_{N}})$
as a trap with depth $N$, i.e.,
\[
\xi'(x)=\begin{cases}
1 & \text{if}\quad x\in\mathbb{J}_{i},\quad1\le i\le b_{N},\\
-N & \text{otherwise}.
\end{cases}
\]
Starting from $\xi'$, since the traps completely isolate the intervals
$\mathbb{J}_{1},\dots, \mathbb{J}_{b_{N}}$, denoting by $\mathcal{T}_{\rm tr}^{{(i)}}$
the SWT transience time on each interval $\mathbb{J}_{i}$,
\[
\mathbb{Q}_{\xi'}\left[\mathcal{T}_{{\rm tr}}\le C_{\rm LB}\log^{3}N\right]=\prod_{i=1}^{b_{N}}\mathbb{Q}_{\xi'}\left[\mathcal{T}_{\rm tr}^{{(i)}}\le C_{\rm LB}\log^{3}N\right].
\]

It remains to give an upper bound for each probability in the right-hand
side. As emphasized at the end of Section \ref{sec4.1}, we may understand
the SWT dynamics on $\mathbb{J}_{i}$ as $|\mathbb{J}_{i}|$ (dependent) simple
random walks with rate $1$ to both left and right. Considering only the
particle at the center and noting that $|\mathbb{J}_i| \ge \Lambda' \log N$, Lemma \ref{lem:SRW}-(2) with $M = \lceil \Lambda' \log N / 2 \rceil$ implies that
\[
\mathbb{Q}_{\xi'}\left[\mathcal{T}_{\rm tr}^{{(i)}}\le C_{\rm LB}\log^{3}N\right]\le1-\frac{1}{2}\exp\left(-\frac{\pi^{2}C_{\rm LB}\log^{3}N}{ (\Lambda' \log N )^2 }\right).
\]
Combining all the displayed observations, we obtain that
\[
\mathbb{P}_{\eta}^{N}\left[\mathcal{T}_{{\rm tr}}\le C_{\rm LB}\log^{3}N\right]\le\left(1-\frac{1}{2}\exp\left(-\frac{\pi^{2}C_{\rm LB}\log^{3}N}{ (\Lambda' \log N )^2 }\right) \right)^{b_{N}}.
\]
Taking $C_{\rm LB}>0$ such that 
\begin{equation}
\frac{\pi^{2}C_{\rm LB}}{(\Lambda')^2}<\frac{1}{2},\label{eq:C-def}
\end{equation}
we have
\[
\lim_{N \to \infty} \exp\left(-\frac{\pi^{2}C_{\rm LB}\log^{3}N}{ (\Lambda' \log N )^2 }\right) \times b_N = \infty,
\]
thus we conclude that
\[
\mathbb{P}_{\eta}^{N}\left[\mathcal{T}_{{\rm tr}}\le C_{\rm LB}\log^{3}N\right]\xrightarrow{N\to\infty}0,
\]
as long as $\eta$ satisfies the property before \eqref{e:proba-consecutive-part}. Equation \eqref{eq:main-LB} follows from this and \eqref{e:proba-consecutive-part}.

\begin{proof}[Proof of Theorem \ref{thm:main}]
The theorem follows immediately from \eqref{eq:main-UB} and \eqref{eq:main-LB}.
\end{proof}

\section{\label{sec5}Polynomial supercritical intervals and dynamical decorrelation}

{The proof of Theorem \ref{thm:main2} also relies on the correspondence between the FEP and the SWT exploited  in the previous section, we refer the reader once again to the brief description of Section \ref{sec:SWTFEPshort} and to Appendix \ref{secA} where this correspondence is explained in details.
Because of this correspondance,} we will study for $K = K_N^\alpha \in (0,N)$ {(defined in \eqref{eq:Kalpha})} the {transience time close to criticality starting from a} uniform state $\mu_{N,K}$ over $\Gamma_{N,K}$ (cf. \eqref{eq:SWT-space}), which is the set of SWT configurations on {$\T_K$ obtained by mapping FEP configurations with $K$ particles (i.e. those satisfying \eqref{eq:mappedGammaNK})}. We abbreviate, for simplicity, 
\[\mu_\alpha :={\mu_{N,K}}= \mu_{N,K_N^\alpha}.\]

In this section, we obtain two key {ingredients} to prove Theorem \ref{thm:main2}.
First, we
%give a dynamical construction of the uniform distribution $\mu_{N,K}$ and thereby
prove local density estimates for typical configurations sampled from $\mu_{N,K}$, which corresponds to the static estimate conducted in Section \ref{sec3} in the subcritical case. {The latter allowed us to identify, from the initial state, barriers-separated regions of the FEP that would \emph{deterministically} never be able to communicate.  Unfortunately,  for the near-critical case such regions do not exist, so that we complement this density estimate with a dynamical \emph{decorrelation} estimate for the SWT which allows us to make sure that the evolution of far enough regions is close to independent.}

\subsection{Finding supercritical boxes under $\mu_\alpha$}
\label{s:supercritboxes}

%\coo{O: Please check the proof below, which replaces the previous 7 pages of computations. The previous proof is commented (hidden) in the tex.}
%\cosw{SW: Looks nice to me, thanks for the effort!}

For $\delta>0$, let us fix $k_N:=\lfloor\ell_NN^{-\delta}\rfloor$, namely a slightly smaller scale than  $\ell_N=N\wedge N^{2\alpha}$ introduced in Theorem \ref{thm:main2}, at which supercritical boxes disappear. We will show that with high probability under $\mu_\alpha$ we can find at least $\mathcal{O}(\log N)$ consecutive supercritical boxes of size $k_N$.

Recall from \eqref{eq:Kalpha} that we let $K=K_N^\alpha = \lfloor \frac N2 - N^{1-\alpha} \rfloor$. Let us start by comparing the uniform measure $\mu_\alpha$ to a simpler product measure. Define $\pi_\alpha$ on the set of SWT configurations $\Z_{\le 1}^{\T_K}$ as 

{\begin{equation}
    \pi_\alpha (\xi)=\prod_{x\in\T_K} \left[ \left(\frac{K}{N}\right)\Big(1-\frac{K}{N}\Big)^{1-\xi(x)} \right] ,\qquad \xi\in\Z_{\le 1}^{\T_K},
\end{equation}
meaning that $\xi(x)$ is $1$ minus a non-negative geometric distribution with parameter $K/N$.}

\begin{lem}\label{l:mupi}
For any $\xi\in\Gamma_{N,K}$,
\begin{equation}
    \frac{\mu_\alpha(\xi)}{\pi_\alpha(\xi)}\le \frac{e^2}{\sqrt{2\pi}}\sqrt{\frac{N(N-K)}{K}}.
\end{equation}
\end{lem}
\begin{proof}
It is straightforward to check that, if $\xi\in\Gamma_{N,K}$ {(the mapped set of SWT configuration is defined in \eqref{eq:SWT-space})},
\begin{equation}
   \pi_\alpha(\xi) =\left(\frac{K}{N}\right)^K\left(1-\frac{K}{N}\right)^{N-K},
\end{equation}
and
\begin{equation}
    \mu_\alpha(\xi)=\frac{1}{\binom{N-1}{K-1}}.
\end{equation}
The lemma follows from these expressions and Robbins' bound \cite{Robbins55}, $n!e^n/n^{n+1/2}\in (\sqrt{2\pi},e]$.
\end{proof}

Now let us define what we are looking for. For $A\subset\T_K$ and $\xi\in \Z_{\le 1}^{\T_K}$, let
\begin{equation}
    S_A=S_A(\xi):=\sum_{i\in A}\xi(i).
\end{equation}

For $q\ge 1$, $A_1,A_2,\dots, A_{q}$ are called \emph{consecutive boxes of size $k_N$} if, for $r\in\{1,\ldots,q\}$,  $A_r=\llbracket a_r, b_r\rrbracket\subset\T_K$, with $b_r-a_r=k_N-1$ and for $r<q$, $b_r+1=a_{r+1}$. 
Denote by
\[
[A_1,\dots,A_q]=\bigsqcup_{r=1}^q A_r
\]
the disjoint union of such consecutive boxes.
Given a SWT configuration $\xi\in\Gamma_{N,K}$, we call a box $ C=[A_1,\dots,A_q]$ a \emph{good $q$-box} if the configuration in each $A_r$ is \emph{supercritical}, i.e.
\begin{equation}
\label{eq:goodbox}
S_{A_r}>0 \qquad \text{for all} \quad 1\leq r\leq q.
\end{equation}

The main result of this subsection is the following estimate.
%Recall that we chose $K=K_N^\alpha=\left\lfloor\frac{N}{2}-N^{1-\alpha}\right\rfloor$.

\begin{lem}
\label{lem:goodbox}
There exists a constant $\beta>0$ such that, if $q=\lfloor \beta \log N\rfloor$,
\begin{equation}
   \lim_{N\rightarrow\infty} \mu_\alpha(\text{there is no good} \enspace q \text{-box in} \enspace \T_K)=0.
\end{equation}
\end{lem}

\begin{proof}
We claim that there exists $c>0$ such that, for $N$ large enough, for any $A\subset\T_K$ a box of size $k=k_N$,
\begin{equation}\label{e:pigoodbox}
    \pi_\alpha(S_A>0)\ge c.
\end{equation}

Let us first show that the lemma follows from this claim. Divide $\T_K$ into consecutive boxes $\T_K=[A_1,A_2,\dots, A_m, A_{m+1}]$, each of same size $|A_j|=k_N$ for $j\leq m:=\lfloor N/k_N\rfloor$, except $A_{m+1}$ which may be smaller.
By Lemma \ref{l:mupi}, {and since $\mu_\alpha$ is supported on $\Gamma_{N,K}$,}
\begin{multline}
\mu_\alpha (\text{there is no good} \enspace q \text{-box in} \enspace \T_K)\\
\le \frac{e^2}{\sqrt{2\pi}}\sqrt{\frac{N(N-K)}{K}} \, \pi_\alpha (\text{there is no good} \enspace q \text{-box in} \enspace \T_K).
\end{multline}
Since $\pi_\alpha$ is product, it is standard to estimate the $\pi_\alpha$-probability in the right-hand side:
\begin{multline}
    \pi_\alpha (\text{there is no good} \enspace q \text{-box in} \enspace \T_K)\\
    \le \pi_\alpha \left(\forall i=0,\dots,\lfloor m/q\rfloor -1, \enspace \exists r\in\{1,\ldots,q\} \enspace \text{s.t.} \enspace S_{A_{iq+r}}\le 0\right)\\
    =\left[1-\pi_\alpha (S_{A_1}>0)^q\right]^{\lfloor m/q\rfloor}
    \le C \exp\left(-\frac{N c^q}{k_Nq}\right)\le C' \exp\left(-\frac{N^\delta N^{\beta\log c}}{\beta\log N}\right).
\end{multline}
Here, $C,C'>0$ are constants independent of $N$. The conclusion follows by choosing $\beta>0$ such that $\delta+\beta\log c>0$.

It remains to prove the claim \eqref{e:pigoodbox}. Note that under $\pi_\alpha$, the variables $(1-\xi(x))_{x\in\T_K}$ are distributed as i.i.d. ($\mathbb N_0$-valued) geometric random variables of parameter ${K/N \simeq \frac{1}{2}-N^{-\alpha}}$ (and expectation $\frac{N-K}{K}\simeq 1+4N^{-\alpha})$. Thus,
\begin{align*}
    \pi_\alpha (S_A\le0)&=\pi_\alpha \Big(\sum_{i=1}^{k_N}(1-\xi(i))\ge k_N\Big)\\
    &=\pi_\alpha \left(\frac{1}{\sqrt{k_N}}\sum_{i=1}^{k_N}\left[1-\xi(i)-\frac{N-K}{K}\right]\ge \sqrt{k_N}\frac{2K-N}{K}\right).
\end{align*}
Note that our choice of $k_N$ implies $\sqrt{k_N}\frac{2K-N}{K}\xrightarrow{N\to\infty} 0$. It is also straightforward to check that the geometric variables satisfy the hypotheses of Lindeberg's central limit theorem \cite[Theorem 27.2]{billingsley}, {meaning the right-hand side above converges to $1/2$, which proves \eqref{e:pigoodbox}}.
\end{proof}

\subsection{\label{sec5.2}Dynamical decorrelation}
Next, we prove an estimate on the propagation of information in the system, in order to show a dynamical decorrelation property for the dynamics in distant regions.
{We denote by $\Q_\xi$ the distribution of the SWT  started from configuration $\xi\in\Z_{\le 1}^{\T_K}$.}

\begin{lem}
\label{lem:decorrelation}
Fix a segment $A \subset \T_K$, and for $\ell>0$, define its enlarged set
\[
A_\ell:=\{i\in \T_K : d(i,A)\leq \ell\}.
\]
Then, there exist a universal constant $c>0$ and a coupling {$\widehat{\Q}_{\xi,A_\ell}$} between the SWT $\{\xi(t)\}_{t\ge 0}$ distributed as $\Q_\xi$  and another SWT $\{\xi'(t)\}_{t\ge0}$ on the segment $A_\ell$ with \emph{closed boundary conditions} started from $\xi_{|A_\ell}$ such that
\[
{\min_\xi \widehat{\Q}_{\xi,A_\ell}} \left(\xi_{|A}(t)=\xi'_{|A}(t), \enspace \forall t\in [0, T] \right) \ge 1 - 4Te^{-c \ell^2/T} - e^{-cT} .
\]
\end{lem}

\begin{proof}
Mark the two edges $e_1$ and $e_2$ that link $A_\ell$ to $\T_K\setminus A_\ell$. We assume that the SWT dynamics on $\T_K$ is ruled by a collection $(\mathscr{C}_e)_{e\in \T_K}$ of rate $1$ independent Poisson clocks, and we define, for $t\ge 0$,
\[
k_t := \big| \{(\mathscr{C}_{e_1}\cup \mathscr{C}_{e_2})\cap [0,t]\} \big|
\]
as the number of rings that occurred until time $t$ on the marked edges. Let $t_k$ be the time of the $ k$-th ring of the marked edges, meaning the time of the $k$-th jump of the process $(k_t)_{t\geq 0}$. %For simplicity, let $k_0 := 0$.

For any integer $k\geq 0$, we then define $\xi^{(k)}$ to be the SWT-like process driven by the Poisson clocks $(\widetilde{\mathscr{C}}_e)_{e\in \T_K}$, where 
\[
\widetilde{\mathscr{C}}_e:=
\begin{cases}
\mathscr{C}_e & \text{for} \quad e\neq e_1, e_2 ,\\
\mathscr{C}_e \cap [0,t_k] & \text{for} \quad e\in\{ e_1, e_2\}.
\end{cases}
\]
 In other words, in $\xi^{(k)}$, we let the marked edges ring $k$ total times, after which we close them both so that no particle can jump across. In particular, $\xi^{(0)}_{|A_\ell}$ and $\xi^{(0)}_{|A_\ell^c}$ are two SWT on the closed segments $A_\ell$ and $A_\ell^c$ evolving independently. Note that $\xi^{(k)}$ is not a Markov process, since the rate at which the marked edges are crossed depends on the whole past of the process, however  $\{(\xi^{(k)}_t, k_t) \}_{t \ge 0}$ is.

We now want to obtain some control, for $k\geq 0$,  over the discrepancies between $\xi^{(k)}$ and $\xi^{(k+1)}$.  Up to $t_k$, both processes are identical, however on the event $R$ that the ring at time $t_{k+1}$ makes a particle jump across one of the marked edges, two discrepancies (with opposite signs) respectively appear in $A_\ell$ and $ A_\ell^c$ between
$\xi^{(k)}$ and $\xi^{(k+1)}$.  After time $t_{k+1}$, the two edges separating $A_\ell$ from $A_\ell^c$ close down, and we can define on $R$, in a unique way, the trajectory $X_t$ of the discrepancy in $A_\ell$, as
\[
X_t={\rm argmax} \left\{ i \in A_\ell : \left|\xi^{(k)}_i(t)-\xi^{(k+1)}_i(t) \right| \right\} , \qquad t \ge t_{k+1}.
\]
With the above construction, the discrepancy's position evolves as an active/inactive random walk. More precisely, a single discrepancy between two SWT can be of two types: either one of the processes contains a particle where the other has an empty site, which we call the \emph{particle type}, or both processes take non-positive values at the site of the discrepancy, meaning one has a $1$-deeper trap than the other, which we call the \emph{trap type}. Let $\sigma(t)\in \{0,1\}$  indicate whether at time $t$ the discrepancy at $X_t$ is of trap type $(\sigma(t)=0)$ or particle type ($\sigma(t)=1$). With these notations, the process $X_t$ jumps at rate $1$ to each of its nearest neighbors while $\sigma(t)=1$, whereas while $\sigma(t)=0$ it remains stuck until enough particles reach $X_t$ (in both processes) to make the discrepancy change its type. As a consequence, $X_t$ jumps symmetrically at rate at most $1$, and we can write for some universal constant $c_1>0$, 
\begin{equation}
\label{eq:LDRWX}
\Prob \left( \sup_{ t \in [t_{k+1} , T]} \left|X_t-X_{t_{k+1}}\right| \geq \ell \right) \leq e^{-c_1 \ell^2/(T-t_{k+1})}\leq e^{-c_1 \ell^2/T},
\end{equation}
by Doob's inequality and a standard large deviations bound on symmetric random walks. 

Furthermore, by construction, at the time the discrepancy appears we have $d(A,X_{t_{k+1}})=\ell$, so that
\begin{equation}
\label{eq:distell}
\sup_{ t \in [t_{k+1} , T]} \left| X_t-X_{t_{k+1}} \right| < \ell \qquad  \text{implies} \qquad  \xi^{(k+1)}_i(t)=\xi^{(k)}_i(t), \quad \forall i\in A, \quad \forall t\leq T .
\end{equation}
Making $k$ vary, we denote by $X^{(k)}$ the trajectory of the discrepancy between $\xi^{(k+1)} $ and $ \xi^{(k)}$ in $A_\ell$. We set by default $X^{(k)}_t=x^{\star}$ if no discrepancy was created at the time of the ring, where $\{x^\star\}=e_1\cap A_\ell$ is an arbitrary discrepancy creation point. Define two events as  {follows}
\begin{itemize}
\item $E_1$: There are less than $4T$ rings before time $T$, i.e. $k_T\leq 4T$, so that in particular $\xi\equiv \xi^{(4T)}$ on $\T_K\times [0,T]$.
\item $E_2$: None of the first $4T$ discrepancies' trajectories has managed to reach $A$ before time $T$, meaning
\[
\sup_{ k \in \llbracket 0,4T-1 \rrbracket} \sup_{ t \in [t_{k+1},T]} \left| X^{(k)}_t-X^{(k)}_{t_{k+1}} \right| < \ell.
\]
\end{itemize}
Then, according to equation \eqref{eq:distell}, on $E_1\cap E_2$, the identity
$\xi^{(k+1)}_i(t)=\xi^{(k)}_i(t)$ holds for any $ i\in A,\; t\leq T,\; k\leq 4T-1$, so that in particular, 
\[\xi_i(t)=\xi^{(0)}_i(t) \qquad \text{on} \quad  A \times [0,T].\]
By a crude large deviations estimate, since both marked edges' clocks ring at rate $1$, the probability of $E_1^c$ is bounded by $e^{-c_2T}$ for some universal constant $c_2 > 0$, whereas the probability of $E_2^c$ is, by union bound and \eqref{eq:LDRWX}, bounded above by $4Te^{-c_1\ell^2/T}$. Choosing $c=c_1\wedge c_2$ proves the lemma.
\end{proof}

For $A\subset \T_K$ and $T>0$, define 
\[
\mathscr{F}_{A,T} := \sigma ( \{\xi_i(t): i\in A, \enspace t\in [0,T]\} )
\]
as the $\sigma$-algebra generated by the configuration observed in $A\times [0,T]$.
As a straightforward consequence of Lemma \ref{lem:decorrelation}, we have the following decorrelation estimate between  $q \ge 1$ boxes.
\begin{cor}
\label{cor:decorrelation}
For $q \ge 1$, fix $q$ segments $A_1,\dots , A_q \subset \T_K$,  that satisfy $d(A_n,A_m)> 2\ell$ for any $n\neq m$. Then, there exists a universal constant $c>0$ such that for any family of events $\{E_n\in\mathscr{F}_{A_n,T}\}_{1 \le n \leq q}$, and initial configuration $\xi$,
\[
\left| \Q_{\xi}\left(\bigcap_{n=1}^q E_n\right)-\prod_{n=1}^q\Q_{\xi}(E_n) \right| \leq 4qTe^{-c\ell^2/T}+qe^{-cT}.
\]
\end{cor}
We will not prove this result, since its proof is completely identical to that of Lemma \ref{lem:decorrelation}: we couple the process with one where marked edges separating the enlarged sets
\[
A_{n,\ell} := \{i\in \T_K, \enspace d(i,A_n)\leq \ell\}, \qquad 1 \le n \le q,
\]
have been closed down, and if there are less than $4qT$ discrepancies created before time $T$, and if none of them reaches one of the $A_n$ for all $1 \le n \le q$, both processes are identical on each $A_n$ up to time $T$.
Finally, the events $E_n$ for $1 \le n \le q$ are independent if all marked edges are closed down initially.

\section{\label{sec6}Proof of Theorem \ref{thm:main2}}

According to \eqref{eq:trans-corr-unif}, by abbreviating $\mathbb{Q}_\alpha := \mathbb{Q}_{\mu_{N,K_N^\alpha}}^N$, to prove Theorem \ref{thm:main2} it is enough to prove it for the SWT, namely,
\begin{equation}\label{eq:main2-wts}
\lim_{N\to\infty}\Q_\alpha \left( \ell_N^2N^{-\varepsilon} < \mathcal{T}_{\rm tr}^N < \ell_N^2N^{\varepsilon} \right) = 1 \qquad \text{for any} \quad \varepsilon > 0 .
\end{equation}
{We are now in a position to prove this estimate, namely} the lower bound in Section \ref{sec6.1} and the upper bound in Section \ref{sec6.2}.

\subsection{\label{sec6.1}Lower bound in Theorem \ref{thm:main2}}
First, we prove the lower bound 
\begin{equation}
\label{eq:lowerbound}
\lim_{N\to\infty}\Q_\alpha \left( \mathcal{T}_{\rm tr}^N > \ell_N^2N^{-\varepsilon} \right) = 1
\end{equation}
for the freezing time. Let $G$ be the $\xi(0)$-measurable event that there exists a good $q$-box (in the sense of \eqref{eq:goodbox}, with $q=\beta\log N$ given by Lemma \ref{lem:goodbox}) somewhere in the system.  If there are more than one good $q$-boxes, we choose one arbitrarily, e.g. the closest to the origin starting from the right, that we denote by $C=[A_1,\dots, A_q]$. According to Lemma \ref{lem:goodbox}, {such a box exists with high probability,} and to prove \eqref{eq:lowerbound} we can consider only what happens on $G$, meaning we will prove that for any $\varepsilon >0$, 
\begin{equation}
\label{eq:LBgoodbox}
\lim_{N\to\infty}\Q_\alpha \left( \mathcal{T}_{\rm tr}^N > \ell_N^2N^{-\varepsilon} \enspace \text{and} \enspace G \right) = 1 .
\end{equation}
On $G$, the (random) supercritical boxes $A_1,\dots, A_q$ are well defined, and we let $A_0:=\T_K\setminus C$. We keep track of the trajectories of each particle in the system, and  define for $r, r'\in\llbracket 0,q\rrbracket$ the total current $J_{r, r'}(t)$, up to time $t$, going from $A_r$ to $A_{r'}$, meaning that $J_{r, r'}(t)$ is the number of particles which were initially in $A_r$, and at time $t$ are located in $A_{r'}$. With the convention $A_{q+1}=A_0$, we then define for $0\leq r\leq q$ 
\[J_r(t)=J_{r, r+1}(t)-J_{r+1, r}(t) \qquad \text{and} \qquad S_r(t)=S_{A_r}(t)=\sum_{i\in A_r}\xi_i(t).\]
Then, for any time $t$, one of two things must  happen:
\begin{enumerate}
\item [i)] either a particle in the system has managed to travel a distance larger than $k_N$ before time $t$, (we call $F_t$ this event) or
\item [ii)] for any $0\leq r\leq q$, 
\begin{equation}\label{eq:Sr-Jr}
S_r(t)=S_r(0)+J_{r-1}(t)-J_r(t).
\end{equation}
\end{enumerate}
This claim is a simple consequence of the fact that each of the $A_r$'s is of size $k_N$ or larger (for $A_0$), so that if the first condition fails, conservation of mass implies that any mass lost or gained must have come from one of the neighboring boxes. Choose now $t=t_N:=\ell_N^2N^{-\varepsilon}$, by union bound on the at most $K$ particles,  and choosing $\delta=\varepsilon/4$ to define $k_N$, we obtain via a simple large deviations estimate that
\[
\Q_\alpha ({F_{t_N}})=\mathcal{O} \left( K e^{-c k_N^2/t_N} \right) = \mathcal{O} \left( N e^{-c N^{\varepsilon/2}} \right),
\]
where $c>0$ is a universal constant.
After the freezing time $\mathcal{T}_{\rm tr}^N$, we must have $S_r(t)\leq 0$ for any $0\leq r\leq q$. As a consequence, to prove \eqref{eq:LBgoodbox}, with the previous estimate, it is enough to show that
\begin{equation}
\label{eq:GFcsupercrit}
\lim_{N\to\infty}\Q_\alpha \left(G \cap {F_{t_N}^c} \cap \left\{ \max_{1\leq r\leq q} S_r(t_N)>0 \right\} \right)=1,
\end{equation}
meaning that with probability tending to one, at time $t_N$ one of the boxes $A_r$ is still strictly supercritical.
But since they are all initially supercritical, on $F^c$, to prove the latter, it is enough to find $1\leq r<r'\leq q$ such that 
\begin{equation}
\label{eq:oppositecurrents}
J_{r}(t_N)\geq 0 \qquad \text{and} \qquad J_{r'}(t_N)\leq 0.
\end{equation}
Indeed, the boxes in between these two currents were initially strictly supercritical, and if the previous condition holds, at least one of them remained strictly supercritical up to time $t_N$ due to \eqref{eq:Sr-Jr}, and in particular cannot have frozen.

We now only need to prove that \eqref{eq:oppositecurrents} holds with probability tending to one. To do so, we split $C$ into groupings of four $A_j$'s, namely for $1\leq r \le q/4$ \[\widetilde{A}_r=[A_{4r-3},A_{4r-2},A_{4r-1},A_{4r}],\] and  we denote by $\widetilde{J}_r(t)=J_{4r-2}(t)$ the current crossing the middle of $\widetilde{A}_r$ up to time $t$.  Define the events
\[
G_r^+= G\cap {F_{t_N}^c}\cap \left\{ \widetilde{J}_r(t_N)\geq 0 \right\} \qquad \text{and}\qquad
G_r^-= G\cap {F_{t_N}^c}\cap \left\{ \widetilde{J}_r(t_N)\leq 0 \right\}.
\]
Note that $G_r^+\cup G_r^-=G\cap F^c$, and by symmetry $\Q_\alpha (G_r^+)=\Q_\alpha (G_r^-)$. This implies 
\begin{equation}
    \Q_\alpha (G_r^\pm)\ge \frac{1}{2}\Q_\alpha (F^c)\ge\frac{1}{4}
\end{equation}
for $N$ large enough, by Lemma \ref{lem:SRW}.
%Note that on the event $G \cap F^c$ the initial distribution of particles in the $\widetilde{A}_r$'s is invariant under symmetry w.r.t. the center of the box. Because of that, by the decorrelation estimate stated in Lemma \ref{lem:decorrelation}, applied to the sets $D_r:=[A_{4r-2},A_{4r-1}]$, to the time  $T=t_N$ and $\ell=k_N$, w.h.p. what happens inside $D_r$ does not depend on the distribution outside of $\widetilde{A}_r$, and is therefore symmetric w.r.t. the center of the box, thus yielding that for any $1\leq r\leq q/4$
%\[
%\Q_{N,K} \left( G_r^\pm \right) \geq \frac12+\mathcal{O} \left( t_Ne^{-c N^{\varepsilon/2}} \right) \geq \frac14.
%\]
Furthermore, since the current $\widetilde{J}_r(t_N)$ is measurable w.r.t. the trajectory of the configuration in $D_r := A_{4r-2} \cup A_{4r-1}$ (actually, it is measurable w.r.t. the trajectory in the two sites at the center of $\widetilde A_r$), we can apply Corollary \ref{cor:decorrelation} to $\ell=k_N$, $T=t_N$, the boxes $(D_r)_{r\leq q/4}$, and to the events $E_r=G_r^\pm$. We  obtain that the probability to find a non-negative current in the first half of the boxes and a non-positive one in the second half is lower bounded by 
\[
1- 2 \left( \frac34 \right)^{q/4} + \mathcal{O} \left( qt_Ne^{-ck_N^2/t_N}+q e^{-ct_N} \right) \xrightarrow{N \to \infty} 1,
\]
which proves that \eqref{eq:oppositecurrents} occurs with probability tending to one, and therefore proves \eqref{eq:GFcsupercrit}.

\subsection{\label{sec6.2}Upper bound in Theorem \ref{thm:main2}}
Finally, we consider the upper bound.
We start by treating the case of large $\alpha\geq 1/2$. In this case, we want to prove that the freezing time is less than $N^{2+\varepsilon}$ with probability tending to $1$. The uniform estimate given in \cite[Theorem 1.1]{EM24} implies that for any diverging sequence $t_K,$ and starting from any initial distribution $\sigma$ of the SWT on $\T_K$, we have
\[
\limsup_{N\to\infty} \Q_\sigma \left( \mathcal{T}_{\rm tr}^N \geq t_K K^2\log K \right) = 0,
\]
which proves in particular the upper bound for $\alpha\geq 1/2$ if one chooses 
\[t_K=N^{2+\varepsilon}/ (K^2\log K) \geq C N^{\varepsilon}/\log K\to\infty \qquad \text{as} \quad N \to \infty.\]

We now consider the case of small $\alpha<1/2$, where we are far enough away from criticality to affect the freezing time. Our objective is to show that for any $\varepsilon>0$,
\[
\lim_{N\to\infty}\Q_{N,K} \left( \mathcal{T}_{\rm tr}^N<\ell_N^2N^{\varepsilon} \right) = 1 .
\]
In a transient configuration $\xi \in \Gamma_{N,K}$, for any site $i$ occupied by a particle, we denote by $a_i<i$, $b_i>i$ its two closest traps to its left and right, meaning that the configuration $\xi$ on $\{a_i+1, \dots, b_i-1\}$ contains no traps, and $\xi_{a_i},\xi_{b_i}<0$. Then, we can write  for any $t$, by translation invariance of the measure $\Q_\alpha$ and a union bound over the positions of the particles and the associated closest traps,
\begin{align*}
\Q_\alpha (\xi(t) \enspace & \text{is transient} ) \leq K^3 \sup_{a,b\in \T_K} \Q_\alpha (\xi_0(t)=1, \enspace  a_0=a, \enspace b_0=b)\\
& \leq K^3 \sup_{a,b\in \T_K}|b-a| \, \Q_\alpha ( [ a+1,b-1] \enspace \text{was initially supercritical}) \, p_{a,b}(t),
\end{align*}
where $p_{a,b}(t)=e^{-ct/(b-a)^2}$ is an upper bound of the probability that a random walk remains in the segment $[ a+1,b-1 ] \subset \mathbb T_K$ up to time $t$, {which it must have because it was constrained by the two traps.}

\begin{itemize}
\item For $|b-a| \leq \ell_N N^{\varepsilon/4}$, we have $p_{a,b}(t_N)=\mathcal{O}(e^{-cN^{\varepsilon/2}})$, so that the corresponding contribution above vanishes.
\item For $|b-a|>\gamma_N := \ell_N N^{\varepsilon/4}$, we define $A= [ a+1,b-1 ]$
and argue as done in the proof of Lemma \ref{lem:goodbox} that there exists a global constant $C>0$ such that
\[
\mu_\alpha (S_A\ge 0) \le C \sqrt{N} \, \pi_\alpha (S_A \ge 0) = C \sqrt{N} \, \pi_\alpha \left( \frac1{\gamma_N} \sum_{i=1}^{\gamma_N} \left[ 1 - \xi(i) - \frac{N-K}K \right] \ge \frac{2K-N}K \right) .
\]
Recall that $1-\xi(i)$, $i \in \mathbb T_K$, are i.i.d. geometric random variables of parameter ${K/N}$. By the large deviations estimate for geometric random variables given in \eqref{eq:LDP-geo}, since ${K/N} \in (\frac13,\frac23)$ and $\frac{2K-N}K \in (0,1)$ for large enough $N$, we have
\[
\pi_\alpha \left( \frac1{\gamma_N} \sum_{i=1}^{\gamma_N} \left[ 1 - \xi(i) - \frac{N-K}K \right] \ge \frac{2K-N}K \right) \le e^{- C \gamma_N \left( \frac{2K-N}K \right)^2},
\]
where $C>0$ is a global constant. As $\gamma_N (\frac{2K-N}K)^2 \simeq N^{\epsilon/4}$, we have proved that 
\[
\limsup_{N\to\infty} \, K^3 \sup_{\substack{a, b \in \mathbb{T}_K \\ |b-a|>\ell_N N^{\varepsilon/4} } }|b-a| \, \Q_\alpha ([ a+1,b-1 ] \enspace \text{was initially supercritical})=0,
\]
as wanted.
\end{itemize}
%and use once again the upper bound of Lemma \ref{lem:boundneghyper} to obtain that
%\begin{equation*}
%\mu_{N,K}(S_A\geq 0)=\mu_{N,K}(P_A\leq k)\leq C N^2 \sum_{p=0}^k \exp\Big\{-(p+k-1)H(\rho_A\mid \bar\rho)\Big\} ,
%\end{equation*}
%\end{itemize}
%where for $0\leq p\leq k$, 
%\[
%\rho_A=\rho_A(p):=\frac{k-1}{p+k-1}\geq\frac{1}{2}-\frac{1}{2k} \qquad \text{and} \qquad  \bar\rho=\frac{K-1}{N-1} \sim \frac 12 -N^{-\alpha}.
%\]
%Since $\ell_N=N^{2\alpha}$, we have in particular $k\gg N^{\alpha}$, therefore for any  $0\leq p\leq k$, $|\rho_A-\bar\rho|\geq N^{-\alpha}/3$, which yields by convexity of $H$ that for some constant $c>0$,
%\[H(\rho_A\mid \bar\rho)\geq cN^{-2\alpha}.\]
%As a consequence, $kH(\rho_A\mid \bar\rho)\geq cN^{\varepsilon/4}$, and we have proved that 
%\[
%\limsup_{N\to\infty} \, (K_N^\alpha)^3\sup_{\substack{a, b \in \mathbb{T}_K \\ |b-a|>\ell_N N^{\varepsilon/4} } }|b-a| \, \Q_{N,K}(\rrbracket a,b\llbracket \enspace \text{was initially supercritical})=0,
%\]
%as wanted.
Combining these two cases concludes the proof of the upper bound part of Theorem \ref{thm:main2}.

\begin{proof}[Proof of Theorem \ref{thm:main2}]
The bijection between FEP and SWT, in particular the identity \eqref{eq:trans-corr-unif}, implies that Theorem \ref{thm:main2} holds true if we prove that \eqref{eq:main2-wts} holds. This is already verified in Sections \ref{sec6.1} (lower bound) and \ref{sec6.2} (upper bound).
\end{proof}

\begin{acknowledgement*}
SK and SL would like to {thank} Universit\'e Claude Bernard Lyon 1 (UCBL) and Centre International de Rencontres Math\'ematiques (CIRM) for their warm hospitalities during their stay.
OB and CE were supported by  ANR grant
MICMOV (ANR-19-CE40-0012) of the French National
Research Agency (ANR).
SK was supported by the National Research Foundation of Korea (NRF) grant funded by the Ministry of Science and ICT (No. RS-2025-00518980), the Yonsei University Research Fund of 2025 (2025-22-0133),
and the POSCO Science Fellowship of POSCO TJ Park Foundation.
SL was supported by the National Research Foundation of Korea (NRF) grant funded by the Ministry of Science and ICT (No. RS-2023-NR076621 and RS-2025-23525546).
\end{acknowledgement*}

\appendix

\section{\label{secA}Correspondence between l-FEP and l-SWT}
In this section, we describe the correspondence between the FEP dynamics and the so-called SSEP with traps (SWT), first introduced in \cite[Section 1.1]{EM24}. The key mechanism is to label each particle in a FEP configuration in increasing order and associate each particle with a site in the induced SWT configuration, whose state is determined by the number of empty sites between that particle and the next one. To build the FEP and SWT dynamics jointly, we keep in memory the labels of the FEP particles in $\mathbb{T}_N$.

The following notation will be useful: for $N\in\mathbb N$ and $i\in\{1,\ldots, N\}$, let $\delta_i\in\{0,1\}^{\T_N}$ be the configuration such that
\begin{equation}
    \delta_i(j)=\mathbf{1}_{i=j}\quad \forall j\in\{1,\ldots,N\}.
\end{equation}
Note that $N$ is omitted from the notation, but will be clear from context.

\subsubsection*{Labeled FEP}

\begin{figure}
\begin{tikzpicture}
\begin{scope}[scale=0.6]
\draw (0,0)--(12,0);
\foreach \i in {0,...,12} { \draw[thick] (\i,-0.1)--(\i,0.1); }
\foreach \i in {0,1,4,6,7,8,10} { \draw (\i+0.5,0.5) circle (0.4); }
\draw (0.5,0.5) node{$7$};
\draw (1.5,0.5) node{$1$};
\draw (4.5,0.5) node{$2$};
\draw (6.5,0.5) node{$3$};
\draw (7.5,0.5) node{$4$};
\draw (8.5,0.5) node{$5$};
\draw (10.5,0.5) node{$6$};
\draw[ultra thick] (0,-0.2)--(0,0.2); \draw[ultra thick] (12,-0.2)--(12,0.2);
\draw[red,thick,->] (6.5,1.1) sin (6,1.3) cos (5.5,1.1);
\end{scope}

\draw[thick,-latex] (7.8,0)--(9.4,0); \draw (8.6,0.1) node[above]{$\Upsilon$};

\begin{scope}[shift={(10,0)},scale=0.6]
\draw (0,0)--(7,0);
\foreach \i in {0,...,7} { \draw[thick] (\i,-0.1)--(\i,0.1); }
\foreach \i in {2,3,6} { \draw (\i+0.5,0.5) circle (0.4); }
\foreach \j in {0} { \draw[] (0,-\j)--(0,-\j-1)--(1,-\j-1)--(1,-\j); }
\draw[ultra thick] (0,-0.2)--(0,0.2); \draw[ultra thick] (7,-0.2)--(7,0.2);
\draw[blue,thick,->] (2.5,1.1) sin (2,1.3) cos (1.5,1.1);
\draw (7.1,0) node[right]{, $2$};
\end{scope}

\begin{scope}[shift={(0,-2)},scale=0.6]
\draw (0,0)--(12,0);
\foreach \i in {0,...,12} { \draw[thick] (\i,-0.1)--(\i,0.1); }
\foreach \i in {0,1,4,5,7,8,10} { \draw (\i+0.5,0.5) circle (0.4); }
\draw (0.5,0.5) node{$7$};
\draw (1.5,0.5) node{$1$};
\draw (4.5,0.5) node{$2$};
\draw (5.5,0.5) node{$3$};
\draw (7.5,0.5) node{$4$};
\draw (8.5,0.5) node{$5$};
\draw (10.5,0.5) node{$6$};
\draw[ultra thick] (0,-0.2)--(0,0.2); \draw[ultra thick] (12,-0.2)--(12,0.2);
\draw[red,thick,->] (8.5,1.1) sin (9,1.3) cos (9.5,1.1);
\end{scope}

\draw[thick,-latex] (7.8,-2)--(9.4,-2); \draw (8.6,0.1-2) node[above]{$\Upsilon$};

\begin{scope}[shift={(10,-2)},scale=0.6]
\draw (0,0)--(7,0);
\foreach \i in {0,...,7} { \draw[thick] (\i,-0.1)--(\i,0.1); }
\foreach \i in {1,3,6} { \draw (\i+0.5,0.5) circle (0.4); }
\foreach \j in {0} { \draw[] (0,-\j)--(0,-\j-1)--(1,-\j-1)--(1,-\j); }
\draw[ultra thick] (0,-0.2)--(0,0.2); \draw[ultra thick] (7,-0.2)--(7,0.2);
\draw[blue,thick,->] (3.5,1.1) sin (4,1.3) cos (4.5,1.1);
\draw (7.1,0) node[right]{, $2$};
\end{scope}

\begin{scope}[shift={(0,-4)},scale=0.6]
\draw (0,0)--(12,0);
\foreach \i in {0,...,12} { \draw[thick] (\i,-0.1)--(\i,0.1); }
\foreach \i in {0,1,4,5,7,9,10} { \draw (\i+0.5,0.5) circle (0.4); }
\draw (0.5,0.5) node{$7$};
\draw (1.5,0.5) node{$1$};
\draw (4.5,0.5) node{$2$};
\draw (5.5,0.5) node{$3$};
\draw (7.5,0.5) node{$4$};
\draw (9.5,0.5) node{$5$};
\draw (10.5,0.5) node{$6$};
\draw[ultra thick] (0,-0.2)--(0,0.2); \draw[ultra thick] (12,-0.2)--(12,0.2);
\draw[red,thick,->] (1.5,1.1) sin (2,1.3) cos (2.5,1.1);
\end{scope}

\draw[thick,-latex] (7.8,-4)--(9.4,-4); \draw (8.6,0.1-4) node[above]{$\Upsilon$};

\begin{scope}[shift={(10,-4)},scale=0.6]
\draw (0,0)--(7,0);
\foreach \i in {0,...,7} { \draw[thick] (\i,-0.1)--(\i,0.1); }
\foreach \i in {1,4,6} { \draw (\i+0.5,0.5) circle (0.4); }
\foreach \j in {0} { \draw[] (0,-\j)--(0,-\j-1)--(1,-\j-1)--(1,-\j); }
\draw[ultra thick] (0,-0.2)--(0,0.2); \draw[ultra thick] (7,-0.2)--(7,0.2);
\draw (7.1,0) node[right]{, $2$};
\draw[blue,thick,->] (6.5,1.1) sin (7,1.3) cos (7.5,1.1);
\end{scope}

\begin{scope}[shift={(0,-6)},scale=0.6]
\draw (0,0)--(12,0);
\foreach \i in {0,...,12} { \draw[thick] (\i,-0.1)--(\i,0.1); }
\foreach \i in {0,2,4,5,7,9,10} { \draw (\i+0.5,0.5) circle (0.4); }
\draw (0.5,0.5) node{$7$};
\draw (2.5,0.5) node{$1$};
\draw (4.5,0.5) node{$2$};
\draw (5.5,0.5) node{$3$};
\draw (7.5,0.5) node{$4$};
\draw (9.5,0.5) node{$5$};
\draw (10.5,0.5) node{$6$};
\draw[ultra thick] (0,-0.2)--(0,0.2); \draw[ultra thick] (12,-0.2)--(12,0.2);
\end{scope}

\draw[thick,-latex] (7.8,-6)--(9.4,-6); \draw (8.6,0.1-6) node[above]{$\Upsilon$};

\begin{scope}[shift={(10,-6)},scale=0.6]
\draw (0,0)--(7,0);
\foreach \i in {0,...,7} { \draw[thick] (\i,-0.1)--(\i,0.1); }
\foreach \i in {1,4} { \draw (\i+0.5,0.5) circle (0.4); }
\foreach \i in {0} { \draw (\i+0.5,-0.5) circle (0.4); }
\foreach \j in {0} { \draw[] (0,-\j)--(0,-\j-1)--(1,-\j-1)--(1,-\j); }
\draw[ultra thick] (0,-0.2)--(0,0.2); \draw[ultra thick] (7,-0.2)--(7,0.2);
\draw (7.1,0) node[right]{, $3$};
\end{scope}
\end{tikzpicture}\caption{\label{figA.1}The correspondence between l-FEP configurations in $\widetilde\Sigma_{N,K}$
and l-SWT configurations in $\widetilde\Gamma_{N,K}$, with $N=12$ and $K=7$, obtained along three jumps indicated by red (in l-FEP) and blue (in l-SWT) arrows.}
\end{figure}
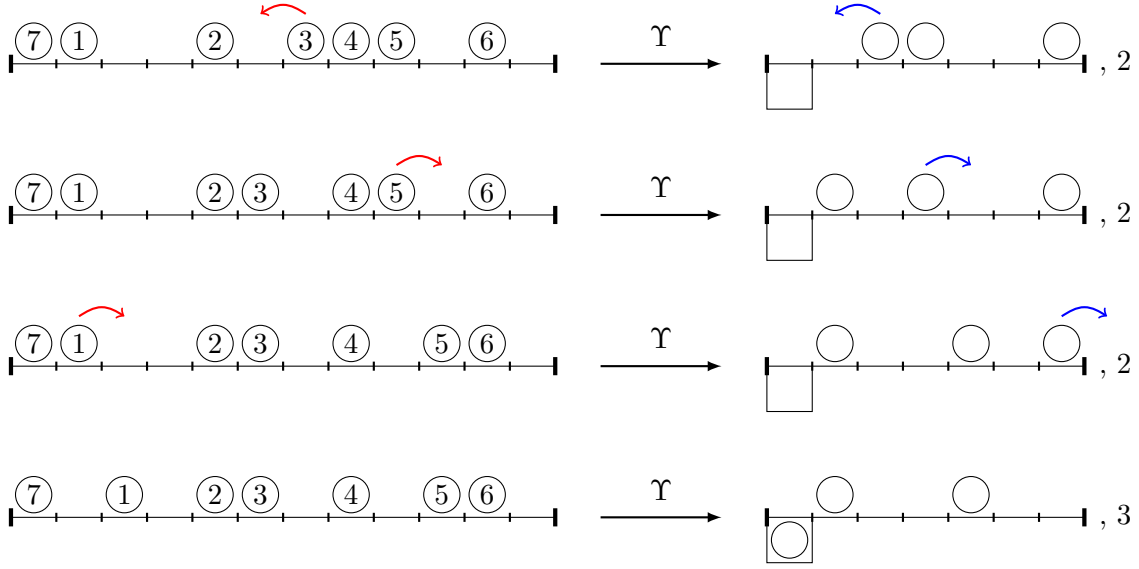

First, let us consider the following labeled version of FEP on $\mathbb T_N$. Fix an integer $K \in [1,N-1]$ which represents the number of particles (we exclude the degenerate cases $K\in\{0,N\}$). Recall \eqref{eq:SigmaNK-def} and define
\begin{equation}\label{eq:tSigmaNK-def}
\widetilde\Sigma_{N,K} := \left\{ X = (X_1 ,X_2,\dots,X_K) \in \mathbb{T}_N^K : X_1, X_2, \dots, X_K \enspace \text{are in increasing order} \right\}.
\end{equation}
Then, define a dynamics $\{\widetilde\eta_t\}_{t \ge 0}$ in $\widetilde\Sigma_{N,K}$, the \emph{labeled} FEP or l-FEP, which is generated by
\begin{align*}
\widetilde{\mathcal{L}}_N F( \widetilde\eta ) =  \sum_{i=1}^K \, \big( & {\bf 1}_{\{{\widetilde\eta}_{i-1}={\widetilde\eta}_i-1,\, {\widetilde\eta}_{i+1} \ne {\widetilde\eta}_i+1\}} ( F({\widetilde\eta}+\delta_i) - F({\widetilde\eta})) \\
& + {\bf 1}_{\{{\widetilde\eta}_{i+1} = {\widetilde\eta}_i+1,\, {\widetilde\eta}_{i-1} \ne {\widetilde\eta}_i-1\}} ( F({\widetilde\eta}-\delta_i) - F({\widetilde\eta}))  \big) \, .
\end{align*}
Here, adding (resp. subtracting) $\delta_i$ to $\widetilde\eta$ indicates moving the $i$-th particle in $\widetilde\eta$ to the right (resp. left). Note that $\widetilde\eta_{K+1} := \widetilde\eta_1$ and $\widetilde\eta_0 := \widetilde\eta_K$.

Defining $\Phi:\widetilde\Sigma_{N,K} \to \Sigma_{N,K}$ as
\begin{equation}
\Phi ( \widetilde\eta ) = \sum_{i=1}^K \delta_{{\widetilde\eta}_i} \in \Sigma_{N,K} \qquad \text{for} \quad \widetilde\eta \in \widetilde\Sigma_{N,K},
\end{equation}
it is easy to verify that the projected process $\{\Phi({\widetilde\eta}_t)\}_{t \ge 0}$ is exactly the FEP in $\Sigma_{N,K}$ (cf. \eqref{eq:SigmaNK-def}). Note that $\Phi$ is surjective but not injective. In fact, $\Phi(\widetilde\eta) = \eta$ for exactly $K$ configurations $\widetilde\eta \in \widetilde\Sigma_{N,K}$.

\begin{notation}\label{not:inv-image}
In the following, for $\eta\in\Sigma_{N,K}$, we use the (abusive) notation $\Phi^{-1}(\eta)$ to indicate the pre-image of $\eta$ such that $\Phi^{-1}(\eta)_1=\min_{i=1,\ldots,N}\Phi^{-1}(\eta)_i$.
\end{notation}

We say that $\widetilde\eta\in\widetilde\Sigma_{N,K}$ is transient (resp.\@ frozen, resp.\@ ergodic) if $\Phi(\widetilde\eta)$ is, and define $\mathcal T^N_{\rm}$ the transience time for the l-FEP accordingly.

\subsubsection*{SSEP with traps (SWT)}
Define 
\begin{equation}\label{eq:SWT-space}
\Gamma_{N,K} := \left\{ \xi \in \mathbb{Z}_{\le 1}^{\mathbb{T}_K} : \sum_{i \in \mathbb{T}_K} \xi(i) = 2K-N \right\},
\end{equation}
where $\mathbb{Z}_{\le 1} = \{n \in \mathbb{Z} : n \le 1 \}$. Then, the SWT dynamics $\{\xi_t\}_{t\ge0}$ in $\Gamma_{N,K}$ is generated by
\[
\mathscr{L}_K g (\xi) = \sum_{\substack{i,j \in \mathbb{T}_K \\ |i-j| = 1}} \, {\bf 1}_{\{\xi(i) = 1, \, \xi(j) \le 0\}} ( g(\xi-\delta_i + \delta_j) - g(\xi)).
\]

We define the transient set for the SWT as
\[
\Gamma_{N,K}^{\rm tr} :=\{\xi\in\Gamma_{N,K} : \xi(i)=1, \enspace \xi(j)<0 \quad  \exists i,j\in\T_K \},% \qquad 
%\Gamma_{N,K}^{\rm rec} := \Gamma_{N,K} \setminus \Gamma_{N,K}^{\rm tr},
\]
and the transience time $\mathcal T^N_{\rm tr}$ accordingly.

\subsubsection*{Labeled SWT}
Let $\widetilde\Gamma_{N,K} := \Gamma_{N,K} \times \mathbb{T}_N$. Define a dynamics $\{{\widetilde\xi}_t\}_{t\ge0}$ in $\widetilde\Gamma_{N,K}$, with $\widetilde\xi_t = (\xi_t , x_t) $, the \emph{labeled} SWT or l-SWT, which is generated by
\begin{align*}
\widetilde{\mathscr{L}}_{K} G (\xi,x) = & \sum_{\substack{i,j \in \mathbb{T}_K \\ |i-j| = 1, \, \{i,j\} \ne \{N,1\}}} \, {\bf 1}_{\{\xi(i) = 1, \, \xi(j) \le 0\}} ( G(\xi-\delta_i + \delta_j , x) - G(\xi,x)) \\
& + {\bf 1}_{\{\xi(K) = 1, \, \xi(1) \le 0\}} ( G(\xi-\delta_K + \delta_1 , x+1) - G(\xi,x)) \\
& + {\bf 1}_{\{\xi(1) = 1, \, \xi(K) \le 0\}} ( G(\xi-\delta_1 + \delta_K , x-1) - G(\xi,x)) .
\end{align*}
Clearly, via the canonical projection $\Psi : \widetilde\Gamma_{N,K} \to \Gamma_{N,K}$, the l-SWT becomes the original SWT. As above, we say that $\widetilde\xi$ is transient iff $\Psi(\widetilde\xi)$ is and define the transience time $\mathcal T^N_{\rm tr}$ accordingly.

\subsubsection*{Bijection between $\widetilde\Sigma_{N,K}$ and $\widetilde\Gamma_{N,K}$}

Define a map $\Upsilon : \widetilde\Sigma_{N,K} \to \widetilde\Gamma_{N,K}$ as follows. For each $\widetilde\eta  \in \widetilde\Sigma_{N,K}$, the corresponding element $\Upsilon({\widetilde\eta}) \in \widetilde\Sigma_{N,K}$ is given as $\Upsilon({\widetilde\eta} )= (\xi_{\widetilde\eta}, \widetilde\eta_1)$ where
\[
(\xi_{\widetilde\eta})_i := \widetilde\eta_i - \widetilde\eta_{i+1} + 2 \in \mathbb{Z}_{\le 1} \qquad \text{for} \quad i \in \mathbb{T}_K.
\]
See Figure \ref{figA.1} for this map. It is clear that $\Upsilon$ is a one-to-one correspondence between $\widetilde\Sigma_{N,K}$ and $\widetilde\Gamma_{N,K}$, via $\Upsilon^{-1} : \widetilde\Gamma_{N,K} \to \widetilde\Sigma_{N,K}$ defined as
\[
\left( \Upsilon^{-1} (\xi,x) \right)_i := x + \sum_{n=1}^{i-1} (2-\xi_n) \qquad \text{for} \quad 1 \le i \le K.
\]
Note that $(\Upsilon^{-1}(\xi,x))_1 = x$.

\subsubsection*{Bijection between l-FEP and l-SWT}

\begin{claim*}
If $\{\widetilde\eta_t\}_{t\ge0}$ is an l-FEP trajectory in $\widetilde\Sigma_{N,K}$, the induced trajectory $\{ \Upsilon(\widetilde\eta_t) \}_{t\ge0}$ in $\widetilde\Gamma_{N,K}$ follows the law of an l-SWT trajectory.
\end{claim*} 

To see this, notice that the jump $\widetilde\eta \to \widetilde\eta + \delta_i$ for $2 \le i \le k$, when $\widetilde\eta_{i-1} = \widetilde\eta_i-1$ and $\widetilde\eta_{i+1} \ne \widetilde\eta_i + 1$ (here $\widetilde\eta_{K+1} = \widetilde\eta_1$), corresponds to the jump
\[
\Upsilon_{\widetilde\eta} \to \Upsilon_{\widetilde\eta} + (-\delta_{i-1} + \delta_i , 0) \qquad \text{when} \quad (\xi_{\widetilde\eta})_{i-1} = 1, \quad (\xi_{\widetilde\eta})_i \le 0 .
\]
The jump ${\widetilde\eta} \to {\widetilde\eta}-\delta_i$ for $2 \le i \le k$, when ${\widetilde\eta}_{i+1} = {\widetilde\eta}_i + 1$ and ${\widetilde\eta}_{i-1} \ne {\widetilde\eta}_i -1 $, corresponds to
\[
\Upsilon_{\widetilde\eta} \to \Upsilon_{\widetilde\eta} + (-\delta_i + \delta_{i-1} , 0) \qquad \text{when} \quad (\xi_{\widetilde\eta})_i = 1, \quad (\xi_{\widetilde\eta})_{i-1} \le 0 .
\]
The jump ${\widetilde\eta} \to {\widetilde\eta} + \delta_1$, when ${\widetilde\eta}_K = {\widetilde\eta}_1 - 1$ and ${\widetilde\eta}_2 \ne {\widetilde\eta}_1 +1 $, corresponds to
\[
\Upsilon_{\widetilde\eta} \to \Upsilon_{\widetilde\eta} + (-\delta_K + \delta_1 , 1) \qquad \text{when} \quad (\xi_{\widetilde\eta})_K = 1, \quad (\xi_{\widetilde\eta})_1 \le 0 .
\]
The jump ${\widetilde\eta} \to {\widetilde\eta} - \delta_1$, when ${\widetilde\eta}_2 = {\widetilde\eta}_1 + 1$ and ${\widetilde\eta}_K \ne {\widetilde\eta}_1 -1 $, corresponds to
\[
\Upsilon_{\widetilde\eta} \to \Upsilon_{\widetilde\eta} + (-\delta_1 + \delta_K , -1) \qquad \text{when} \quad (\xi_{\widetilde\eta})_1 = 1, \quad (\xi_{\widetilde\eta})_K \le 0 .
\]
All types of jumps have rate $1$. This proves the claim.

\subsubsection*{Transience time between the dynamics}

In particular, the transience time in all four processes, FEP, l-FEP, l-SWT, and SWT, are all equivalent via the maps
\begin{equation}\label{eq:corres}
 \Sigma_{N,K} \xleftarrow{\Phi} \widetilde\Sigma_{N,K} \xrightarrow{\Upsilon} \widetilde\Gamma_{N,K} \xrightarrow{\Psi} \Gamma_{N,K}.
 \end{equation}

We may use this correspondence to study the transience time of the FEP via the transience time of the SWT in the following way. It is straightforward to check that, for any $\widetilde\eta\in\widetilde\Sigma_{N,K}$, $\widetilde\eta$ is transient iff $\Phi(\eta)$ is transient iff $\Upsilon(\widetilde\eta)$ is transient iff $\Psi\circ\Upsilon(\widetilde\eta)$ is transient. Consequently (recall Notation \ref{not:inv-image} for defining $\Phi^{-1}(\eta)$ for $\eta \in \Sigma_{N,K} \subset \Sigma_N$),
% Given an initial configuration $\eta \in \Sigma_{N,K} \subset \Sigma_N$, we may choose any ordering of the $K$ particles and label them to obtain $\Phi^{-1}\eta \in \widetilde\Sigma_{N,K}$ (with a slight abuse of notation). Then, the FEP trajectory starting from $\eta \in \Sigma_{N,K}$ can be identified with the SWT trajectory starting from $\Psi \Upsilon \Phi^{-1}\eta \in \Gamma_{N,K}$; more precisely, 
for any Borel set $B \subset \mathbb R$,
\begin{equation}\label{eq:trans-corr}
\mathbb{P}_\eta^N \left[ \mathcal{T}_{\rm tr}^N \in B \right] =  \widetilde{\mathbb{P}}_{\Phi^{-1}\eta}^N \left[ \mathcal{T}_{\rm tr}^N \in B \right]
= \widetilde{\mathbb{Q}}_{\Upsilon \Phi^{-1}\eta}^N \left[ \mathcal{T}_{\rm tr}^N \in B \right] = \mathbb{Q}_{\Psi \Upsilon \Phi^{-1}\eta}^N \left[ \mathcal{T}_{\rm tr}^N \in B \right],
\end{equation}
where the four laws, $\mathbb{P}_\cdot^N$, $\widetilde{\mathbb{P}}_\cdot^N$, $\widetilde{\mathbb{Q}}_\cdot^N$, and $\mathbb{Q}_\cdot^N$, denote the laws of the trajectories following the rules of the FEP, l-FEP, l-SWT, and SWT, respectively.

\subsubsection*{Initial uniform distributions}
Recall from \eqref{eq:nuNK-def} that $\nu_{N,K}$ denotes the uniform distribution on $\Sigma_{N,K}$. Denote by $\widetilde\nu_{N,K}$, $\widetilde\mu_{N,K}$, and $\mu_{N,K}$ the uniform distributions on $\widetilde\Sigma_{N,K}$, $\widetilde\Gamma_{N,K}$, and $\Gamma_{N,K}$, respectively. Then, it is clear  that
\[
\widetilde\nu_{N,K} \circ \Phi^{-1} = \nu_{N,K}, \qquad \widetilde\nu_{N,K} \circ \Upsilon^{-1} = \widetilde\mu_{N,K}, \qquad \text{and} \qquad \widetilde\mu_{N,K} \circ \Psi^{-1} = \mu_{N,K}
\]
(here $\Phi^{-1}$ and $\Psi^{-1}$ applied to measurable sets designate their preimage under $\Phi,\Psi$). Thus, the law of the transience time is preserved along the correspondence for uniform distributions, i.e.,
\begin{equation}\label{eq:trans-corr-unif}
\mathbb{P}_{\nu_{N,K}}^N \left[ \mathcal{T}_{\rm tr}^N \in B \right] =  \widetilde{\mathbb{P}}_{\widetilde\nu_{N,K}}^N \left[ \mathcal{T}_{\rm tr}^N \in B \right]
= \widetilde{\mathbb{Q}}_{\widetilde\mu_{N,K}}^N \left[ \mathcal{T}_{\rm tr}^N \in B \right] = \mathbb{Q}_{\mu_{N,K}}^N \left[ \mathcal{T}_{\rm tr}^N \in B \right] .
\end{equation}

\section{Estimates for simple random walks}\label{a:RW}

\subsection{Proof of Lemma \ref{l7}}\label{a:proof-l7}

%\coo{O: Here is an alternative proof.} \cosw{SW: Great, taken!}
Let us define $k_N:=\lceil\Lambda'\log N\rceil+1$ and
\begin{equation}
    A_i:=\{X_i=0, \enspace X_{i+1} = \dots = X_{i+k_N}=1\}.
\end{equation}
We estimate
\begin{align}
    P \left( \bigcup_{ i=1}^{M-k_N}A_i \right) \ge P\left(\bigcup_{j=0}^{\lfloor\frac{M}{k_N+1}\rfloor-1}A_{1+j(k_N+1)}\right)
    = 1-\left(1-(1-\rho)\rho^{k_N}\right)^{\lfloor\frac{M}{k_N+1}\rfloor-1}.
\end{align}
Since $\rho^{k_N} \ge \rho^{\Lambda' \log N +2} = \rho^2 N^{\Lambda'\log \rho}$, we have the desired estimate as soon as
\begin{equation}\label{eq:c6-def}
\Lambda'|\log\rho| < 1/2.
\end{equation}

\subsection{Proof of Lemma \ref{l2}}\label{a:proof-exp-moment}
Recall the asymmetric random walk $Y$ defined in \eqref{eq:Yn-def}, with $Y_0=0$.
Denote by $H_M$ the first hitting time of $\{-1,M\}$, such that $H_M \uparrow \tau_{-1}$ as $M \to \infty$. Let $g(i) := E_i [e^{\lambda H_M}]$ for $-1 \le i \le M$, such that $\lim_{M \to \infty} g(0) = E_0 [e^{\lambda \tau_{-1}}]$ by the monotone convergence theorem. Then, $g(-1) = g(M) = 1$ and
\[
g(i) = (1-\rho) ( g(i-1) e^\lambda ) + \rho ( g(i+1) e^\lambda ) \qquad \text{for} \quad 0 \le i \le M-1,
\]
via the strong Markov property. Defining
\[
a := \frac{e^{-\lambda} + \sqrt{e^{-2\lambda}-4\rho(1-\rho)}}{2\rho} \qquad \text{and} \qquad b := \frac{e^{-\lambda} + \sqrt{e^{-2\lambda}-4\rho(1-\rho)}}{2(1-\rho)},
\]
we obtain $g(i)-ag(i-1) = b ( g(i+1) - ag(i))$. Solving these equations for $-1 \le i \le M$, if $e^{-2\lambda} > 4\rho(1-\rho)$ then
\[
g(0) = a - \left( a - \frac1b \right) \frac{1 - a^{-M-1}}{1-(ab)^{-M-1}} \xrightarrow{M \to \infty} a - \left( a - \frac1b \right) = \frac1b,
\]
where in the limit we used $a>1$ and $ab>1$. On the other hand, if $e^{-2\lambda} = 4\rho(1-\rho)$ then
\[
g(0) = a - \frac{a \left( 1-a^{-M-1} \right)}{M+1} \xrightarrow{M \to \infty} a = \frac1b,
\]
since in this case $a>1$ and $ab=1$. Simplifying $1/b$, we obtain the result.

\subsection{\label{secB}Tail probability of hitting times}

Here, we state and prove (for completeness) a classical tail probability
estimate of hitting times for the simple random walk on $\mathbb{Z}$.
Roughly, it states that the probability of not escaping an interval
of size $M$ until time $\gamma$ is asymptotically close to
$e^{-\frac{c\gamma}{M^{2}}}$ for some constant $c>0$.
\begin{lem}
\label{lem:SRW}Consider a continuous-time simple random walk $\{x_{t}\}_{t\ge0}$
on $\mathbb{Z}$ with jump rate $1$ on both directions. Denote by
$P_{x}$ the law of $\{x_{t}\}_{t\ge0}$ starting from $x$. For each
$M\in\mathbb{N}$, denote by $\tau_{M}$ the first hitting time of
$\{-M,M\}$. For all $\gamma>0$, the following estimates hold.
\begin{enumerate}
\item 
\[
\sup_{x\in(-M,M)}P_{x}(\tau_{M}>\gamma)\le2\exp\left(-\frac{2\pi\gamma}{9M^{2}}\right).
\]
\item 
\[
P_{0}(\tau_{M}>\gamma)\ge\frac{1}{2}\exp\left(-\frac{\pi^{2}\gamma}{4M^{2}}\right).
\]
\end{enumerate}
\end{lem}

\begin{proof}
(1) Consider the embedded Markov chain $y=\{y_{n}\}_{n\in\mathbb{N}_{0}}$
on $\mathbb{Z}$ such that $x_{t}=y_{N(t)}$ where $\{N(t)\}_{t\ge0}$
is a Poisson process of intensity $2$ and independent of $y$. Denote
by $\widehat{P}_{\cdot}$ and $\widehat{E}_{\cdot}$ the law and the
corresponding expectation, respectively, of the embedded chain $y$.
Note that for any $x\in(-M,M)$, $\beta\in(0,\frac{\pi}{2M})$, and
$n\ge1$, by the Markov property,
\[
\widehat{E}_{x}[\cos(\beta y_{n})\,|\,\mathcal{F}_{n-1}]=\cos(\beta y_{n-1})\widehat{E}_{0}\cos(\beta y_{1})-\sin(\beta y_{n-1})\widehat{E}_{0}\sin(\beta y_{1})=\cos(\beta y_{n-1})\cos\beta.
\]
This implies that $\{\frac{\cos(\beta y_{n})}{\cos^{n}\beta}\}_{n\ge0}$
is a $\widehat{P}_{x}$-martingale. Thus, by the stopping time theorem
applied to $\tau_{M}\wedge K$ for any $K\in\mathbb{N}$,
\[
1\ge\cos(\beta x)=\widehat{E}_{x}\left[\frac{\cos(\beta y_{\tau_{M}\wedge K})}{\cos^{\tau_{M}\wedge K}\beta}\right]\ge\cos(\beta M)\widehat{E}_{x}\left[\frac{1}{\cos^{\tau_{M}\wedge K}\beta}\right],
\]
where the second inequality holds since $y_{\tau_{M}\wedge K}\in[-M,M]$
and $\beta\in(0,\frac{\pi}{2M})$. Sending $K\to\infty$, we get
\[
\widehat{E}_{x}\left[\frac{1}{\cos^{\tau_{M}}\beta}\right]\le\frac{1}{\cos(\beta M)}.
\]
Thus, for $n\in\mathbb{N}_{0}$,
\[
\widehat{P}_{x}(\tau_{M}>n)\le\cos^{n}\beta\widehat{E}_{x}\left[\frac{1}{\cos^{\tau_{M}}\beta}\right]\le\frac{\cos^{n}\beta}{\cos(\beta M)}.
\]
Using this inequality, we estimate as
\begin{equation}
P_{x}(\tau_{M}>\gamma)=\sum_{n=0}^{\infty}P(N(\gamma)=n)\widehat{P}_{x}(\tau_{M}>n)\le\sum_{n=0}^{\infty}e^{-2\gamma}\frac{(2\gamma)^{n}}{n!}\frac{\cos^{n}\beta}{\cos(\beta M)}=\frac{e^{-2\gamma(1-\cos\beta)}}{\cos(\beta M)}.\label{eq:SRW-1}
\end{equation}
Substituting $\cos\beta\le1-\frac{\beta^{2}}{\pi}$ and $\beta=\frac{\pi}{3M}$,
we conclude that
\[
\sup_{x\in(-M,M)}P_{x}(\tau_{M}>\gamma)\le\frac{e^{-2\gamma\pi^{-1}\beta^{2}}}{\cos(\beta M)}=2e^{-\frac{2\pi}{9}\frac{\gamma}{M^{2}}}.
\]
(2) According to the formula in \cite[p243, line -5]{Spi64},
\begin{align*}
\widehat{P}_{0}( & \tau_{M}>n)=\frac{1}{M}\sum_{j=0}^{M-1}(-1)^{j}\cos^{n}\frac{(2j+1)\pi}{2M}\cot\frac{(2j+1)\pi}{4M}\\
 & =\frac{1}{M}\sum_{0\le j\le\frac{M-1}{2}}(-1)^{j}\cos^{n}\frac{(2j+1)\pi}{2M}\left(\cot\frac{(2j+1)\pi}{4M}+(-1)^{n+M-1}\tan\frac{(2j+1)\pi}{4M}\right).
\end{align*}
Note that the second summation in $j$ is an alternating series with
decreasing absolute terms. Using $\cos x\ge1-\frac{x^{2}}{2}$ and
$\cot x\ge\frac{\pi}{4x}$ on $(0,\frac{\pi}{4})$, the first term
therein becomes
\[
\cos^{n}\frac{\pi}{2M}\left(\cot\frac{\pi}{4M}+(-1)^{n+M-1}\tan\frac{\pi}{4M}\right)\ge\left(1-\frac{\pi^{2}}{8M^{2}}\right)^{n}\left(M+\frac{(-1)^{n+M-1}}{M}\right).
\]
Using $\cos x\le1-\frac{x^{2}}{\pi}$ and $\cot x\le\frac{1}{x}$
on $(0,\frac{\pi}{2})$, the second term therein becomes
\[
\cos^{n}\frac{3\pi}{2M}\left(\cot\frac{3\pi}{4M}+(-1)^{n+M-1}\tan\frac{3\pi}{4M}\right)\le\left(1-\frac{9\pi}{4M^{2}}\right)^{n}\left(\frac{4M}{3\pi}+(-1)^{n+M-1}\frac{3\pi}{4M}\right).
\]
Combining these two estimates,
\[
\widehat{P}_{0}(\tau_{M}>n)\ge\frac{1}{2}\left(1-\frac{\pi^{2}}{8M^{2}}\right)^{n}.
\]
Substituting this bound to \eqref{eq:SRW-1}, we conclude that
\[
P_{0}(\tau_{M}>\gamma)\ge\sum_{n=0}^{\infty}e^{-2\gamma}\frac{(2\gamma)^{n}}{n!}\frac{1}{2}\left(1-\frac{\pi^{2}}{8M^{2}}\right)^{n}=\frac{1}{2}e^{-\frac{\pi^{2}\gamma}{4M^{2}}}.
\]
\end{proof}

\section{Large deviations bound for geometric random variables}
In this section, we record a basic large deviations bound for geometric random variables.
Suppose that $X_1, X_2, \dots$ are i.i.d. with geometric distribution of parameter $\rho \in (0,1)$, i.e.,
\[
P(X_i = m) = (1-\rho) \rho^m \qquad \text{for all} \quad m \ge 0.
\]
Note that $E[X_i] = \frac \rho{1-\rho}$. Let $S_n := X_1 + \cdots + X_n$. Then, via Markov's inequality, we estimate
\[
P \left( \frac{S_n}n \ge \frac\rho{1-\rho} + a \right) \le \lambda^{- n \left( \frac\rho{1-\rho} + a \right)} \, E \left[ \lambda^{ S_n } \right]
= \lambda^{- n \left( \frac\rho{1-\rho} + a \right)} \, E \left[ \lambda^{ X_1} \right]^n .
\]
Above, $\lambda >0$. We have for $\lambda\in(0,1/\rho)$,
\[
E \left[ \lambda^{ X_1} \right] = \sum_{m=0}^\infty \lambda^{ m} (1-\rho) \rho^m = \frac {1-\rho}{1- \lambda \rho} ,
\]
thus
\[
P \left( \frac{S_n}n \ge \frac\rho{1-\rho} + a \right) \le \left( \frac{1-\rho}{ \lambda^{\frac\rho{1-\rho} + a} (1- \lambda \rho)} \right)^n .
\]
By choosing $\lambda = \frac{1 + a(1-\rho)/\rho}{1 + a (1-\rho)} \in(0,1/\rho)$, we have $1-\lambda \rho = \frac{1-\rho}{1+a(1-\rho)}$ and, by Bernoulli's inequality, for $a>0$,
\begin{align*}
\lambda^{\frac\rho{1-\rho} + a} = \left( 1 + \frac{a(1-\rho)^2/\rho}{1+a(1-\rho)} \right)^{\frac\rho{1-\rho} + a}
& \ge 1 + \frac{a(1-\rho)^2/\rho}{1+a(1-\rho)} \left( \frac\rho{1-\rho} + a \right)
\\ & = 1 + a(1-\rho) + \frac{a^2 (1-\rho)^3/\rho}{1+a(1-\rho)}.
\end{align*}
Combining the last two inequalities and using $1-x \le e^{-x}$, we conclude that
\begin{align*}
P \left( \frac{S_n}n \ge \frac\rho{1-\rho} + a \right) & \le \left( \frac{1+a(1-\rho)}{ 1+a(1-\rho) + \frac{a^2 (1-\rho)^3/\rho}{1+a(1-\rho)} } \right)^n \\
& \le \exp \left( -n \, \frac{\frac{a^2 (1-\rho)^3/\rho}{1+a(1-\rho)}}{1+a(1-\rho) + \frac{a^2 (1-\rho)^3/\rho}{1+a(1-\rho)}} \right) .
\end{align*}
In particular, if $\rho \in (1/3,2/3)$ and $a \in (0,1)$, then
\begin{equation}\label{eq:LDP-geo}
P \left( \frac{S_n}n \ge \frac\rho{1-\rho} + a \right) \le \exp ( - C n a^2 )
\end{equation}
for a global constant $C>0$.


\begin{thebibliography}{1}

\bibitem{AC25}J. Ayre and P. Chleboun.
Mixing Times for the Facilitated Exclusion Process.
arXiv:2402.18999, 2024.

\bibitem{AGLS23} A. Ayyer, S. Goldstein, J. L. Lebowitz and E. R. Speer.
Stationary states of the one-dimensional facilitated asymmetric exclusion process.
Ann. Inst. Henri Poincar\'e Probab. Stat., 59(2):726--742, 2023.

\bibitem{BBCS18} J. Baik, G. Barraquand, I. Corwin and T. Suidan.
\textit{Facilitated exclusion process}. 
Computation and Combinatorics in Dynamics, Stochastics and Control, 1--35, 2018.

\bibitem{WAFEP-BBS} G. Barraquand, O. Blondel and M. Simon.
Weakly asymmetric facilitated exclusion process.
Electron. J. Probab., 30, No. 128, 2025.

%\bibitem{basu-mohanty} U. Basu and P. K. Mohanty, Active–absorbing-state phase transition beyond directed percolation: A class of exactly solvable models, Phys. Rev. E 79 (2009), 041143.

\bibitem{billingsley} P. Billingsley.
\textit{Probability and measure (3rd ed.)}.
Wiley Series in Probability and Mathematical Statistics. John Wiley \& Sons, New York, 1995.

\bibitem{BESS20} O. Blondel, C. Erignoux, M. Sasada and M. Simon.
Hydrodynamic limit for a facilitated exclusion process.
Ann. Inst. Henri Poincar\'e Probab. Stat., 56(1):667--714, 2020.

\bibitem{BES21} O. Blondel, C. Erignoux and M. Simon.
Stefan problem for a nonergodic facilitated exclusion process.
Probab. Math. Phys., 2(1):127--178, 2021.

\bibitem{DCES} H. Da Cunha, C. Erignoux and M. Simon.
Hydrodynamic limit for an open facilitated exclusion process with slow and fast boundaries. 
Comm. Math. Phys., 407(3), Paper No. 50, 2026. 

\bibitem{EM24} C. Erignoux and B. Massouli\'e.
Cutoff for the transience and mixing time of a SSEP with traps and consequences on the FEP.
Ann. Inst. Henri Poincar\'e Probab. Stat. (to appear), arXiv:2403.20010, 2024.

\bibitem{mappingESZ} C. Erignoux, M. Simon and L. Zhao.
Mapping hydrodynamics for the facilitated exclusion and zero-range processes. 
Ann. Appl. Probab., 34(1B):1524--1570, 2024. 

\bibitem{fluctuationsFEP} C. Erignoux and L. Zhao.
Stationary fluctuations for the facilitated exclusion process. 
Electron. J. Probab., 29, No. 152, 2024. 

\bibitem{GLS} 
S. Goldstein, J. L. Lebowitz and E. R. Speer.
Approach to hyperuniformity of steady states of facilitated exclusion processes.
Journal of Physics: Condensed Matter, 36(34): 345402, 2024.

\bibitem{HT25}
I. Hartarsky and C. Toninelli.
\textit{Kinetically constrained models}.
SpringCham, 2025.

\bibitem{RAS14} F. Rassoul-Agha and T. Sepp\"al\"ainen.
\textit{A Course on Large Deviations with an Introduction to Gibbs Measures}.
American Mathematical Society, 2014.

\bibitem{Robbins55} H. Robbins.
A Remark on Stirling's Formula.
The American Mathematical Monthly, 62(1):26–29, 1955.

\bibitem{RPSV} M. Rossi, R. Pastor-Satorras, and A. Vespignani.
Universality class of absorbing phase transitions with a conserved field.
Phys. Rev. Lett., 85(9):1803--1806, 2000.

\bibitem{Spi64} F. Spitzer.
\emph{Principles of Random Walk}.
Graduate Text in Mathematics, Springer New York, 1964.

\bibitem{golf} Z. Varin. {The golf model on $\mathbb{Z}/n\mathbb{Z}$ and on $\mathbb{Z}$}. Electron. J. Probab. 30: 1-58 (2025).
\end{thebibliography}
\end{document}